\documentclass[11pt]{article}

\usepackage{amsmath,amsthm,amscd,amsfonts,amssymb}
\usepackage{appendix}
\def\cF{\mathcal{F}}
\def\cG{\mathcal{G}}
\newcommand{\cH}{\mathcal{H}}

\def\cL{\mathcal{L}}
\def\cM{\mathcal{M}}

\def\cO{\mathcal{O}}

\def\cR{\mathcal{R}}
\def\cS{\mathcal{S}}
\def\cT{\mathcal{T}}

\def\al{\alpha}
\def\be{\beta}

\def\de{\delta}

\def\De{\Delta}

\def\la{\lambda}

\def\Sig{{\Sigma}}

\newtheorem{thm}{Theorem}
\newtheorem{lem}[thm]{Lemma}
\newtheorem{prop}[thm]{Proposition}

\newtheorem{remark}[thm]{Remark}

\theoremstyle{definition}
\newtheorem{definition}[thm]{Definition}

\newtheorem{problem}[thm]{Problem}

\def \pxn {\frac{\partial}{\partial x_n}}

\def \ni {\noindent}

\def \pxone {\frac{\partial}{\partial x_1}}

\def \pxn {\frac{\partial}{\partial x_n}}

\def \ank {\mathbb A_n^k}
\def \ans {\mathbb A_n^s}

\def \an {\mathbb A_n}

\numberwithin{equation}{section}
\numberwithin{thm}{section}

\begin{document}

\author{L. Stolovitch\thanks{CNRS-Laboratoire J.-A. Dieudonn\'e U.M.R. 7351, Universit\'e de Nice - Sophia Antipolis, Parc Valrose 06108 Nice Cedex 02, France, email : {\tt stolo@unice.fr}. Research of L. Stolovitch was supported by ANR grant ``ANR-10-BLAN 0102'' for the project DynPDE}}

\title{Big denominators and analytic normal forms\\{\large with an appendix by M. Zhitomirskii\thanks{Department of Mathematics, Technion, 32000 Haifa, Israel, email : {\tt mzhi@techunix.technion.ac.il}. Research of M. Zhitomirskii was supported by the Israel Science Foundation grant 1383/07}}}
\maketitle
\begin{abstract}
We study the {\it regular} action of an analytic pseudo-group of transformations on the space of germs of various analytic objects of local analysis and local differential geometry. We fix a homogeneous object $F_0$ and we are interested in an analytic normal form for the whole affine space $\{F_0 + h.o.t.\}$. We prove that if the cohomological operator defined by $F_0$ has the big denominators property and if a formal normal form is well chosen then this formal normal form holds in analytic category.
We also define big denominators in systems of nonlinear PDEs and prove a theorem on local analytic solvability of systems of nonlinear PDEs with big denominators. 
Moreover, we prove that if the denominators grow ``relatively fast", but not fast enough to satisfy the big denominator property, then we have a normal form, respectively local solvability of PDEs, in a formal Gevrey category. We illustrate our theorems by explanation of known results and by new results in the problems of local classification of singularities
of vector fields, non-isolated singularities of functions, tuples of germs of vector fields, local Riemannian metrics and conformal structures.

\end{abstract}


\vfill\break

\tableofcontents


\section{Introduction}
\label{sec-introduction}

This article is mainly concerned with local analytic classification of various analytic objects, like
 tuples of vector fields, tuples of maps, Riemannian metrics, conformal structures, under the actions of various pseudo-groups  of germs of analytic objects, the simplest one is the pseudo-group of local diffeomorphisms.
 We fix ``the leading part" $F_0$ which is usually a well-understood object and which is
 homogeneous, say of degree $q$, with respect to some grading, and we work in the affine space consisting of analytic objects of the
 form $F_0 + {\mathcal V}_{>q}$, where ${\mathcal V}_{>q}$ is a neighborhood of $0$ in the space of germs of analytic objects of order greater then $q$. An element of that space will be denoted by $F_0 + h.o.t.$.
  A natural way to compare $f$ and  $F_0$ is to construct a normal form $\mathcal N\subset \{F_0 + h.o.t.\}$, defined by $F_0$ and  serving for
  the whole affine space  $\{F_0 + h.o.t.\}$. Constructing a formal normal form (on the level of formal power series)
  can be reduced to linear algebra. A formal normal form is enough for certain applications, but
  in many cases we need an analytic normal form and one should deal with the following question:
   under which conditions the chosen formal normal $\mathcal N$ holds in analytic category, i.e.
   any analytic object $f\in \{F_0 + h.o.t.\}$ can be brought to $\widetilde f\in \mathcal N$
   by an {\it analytic} transformation of the given pseudo-group, and consequently $\widetilde f$ is also
   analytic.   This question
   is precisely the problem we want to address in this article.

\medskip

In this paper we define the big denominators property in general case (for an arbitrary local classification problem),
we show that this property holds in a number of problems, and we prove that if we have big denominators
and a formal normal form is well-chosen then this formal normal form holds in analytic category.

\medskip

The classical obstacles for transition from formal to analytic category are {\it small divisors} such as those encountered in celestial mechanics or local dynamical systems. There are other obstacles. For example,
in the problem of local classification of vector fields of the form $\dot x = Ax + h.o.t. $ with a fixed matrix $A$
there are the following cases: (a) small divisors, (b) no small divisors,
but there are infinitely many resonant relations and (c) no small divisors, the tuple of the eigenvalues of $A$ belongs to the Poincar\'e domain (and consequently there are not more than a finite number of resonant relations).
The resonant, or Poincar\'e-Dulac formal normal form does not hold in analytic category not only in case (a), but also, as it was proved by A. Brjuno, in  case (b),
 see \cite{Encyclo1,Bruno}. In case (c) it holds in analytic category. One of the explanation of this classical theorem,
 the explanation which is the {\it starting point for this paper}, is as follows:
  instead of having small divisors, in case (c) one has {\it big denominators}.

\medskip

We claim that there are many other significant local classification problems where Theorem \ref{main1}
works and allows obtain new results. The main idea of big denominators is that the {\it big denominators property will overcome the factorial divergence} provided by the derivatives  in the nonlinear equations expressing the analytic equivalence of an object to
an object having a formal normal form.

\medskip

In section \ref{sec-big-denom-classif-problems} we define a very wide class of local classification problems
we deal with, we present a simple way for constructing a formal normal form, and we define the big denominator
property. Our first main Theorem \ref{main1} (see below) states that in the case of big denominators and uniformly bounded
formal normal form, this normal form holds in the analytic category.

\medskip

In section \ref{sec-PDEs} we formulate our second main theorem, Theorem \ref{main-thm} on local solvability of non-linear systems of
PDEs. We define big denominators for such systems and prove the local analytic solvability in the case of big denominators. We explain that Theorem \ref{main1} is a simple corollary of Theorem \ref{main-thm}.
It means that a ``right place" of Theorem \ref{main1} is the local theory of analytic non-linear PDEs rather than
local classification problems. Nevertheless the applications of Theorem \ref{main-thm} that we know concern
namely local classification problems. For instance, the conjugacy of two germs of vectors fields $X,Z$ by the mean of a germ of a diffeomorphism $\Psi$ is given by $\Psi_*X=Z$ and can written as a non-linear PDE's satisfied by $\Psi$.

\medskip

Theorems \ref{main1} and \ref{main-thm} are proved in section \ref{sec-proofs}.

\medskip

What happens if there are the big denominators, but they are not big enough to overcome the growth of the derivatives? Divergence of the solution is to be expected as it is well known for germs of vector fields. Can this divergence be very wild ? This question is studied in section \ref{sec-Gevrey}. We prove that in this case a formal normal form holds in $\al$-Gevrey category for some $\al >0$ (the loss of growth).  It means that the norm of the homogeneous degree $i$ part of the normalizing transformation grows as $(i!)^{\al}$. This shows that even if the solution diverges, this divergence is not too wild. We recover recent results by Bonckaert-De Maesschalck and also by Iooss-Lombardi. This is the same phenomenon as in problem of nonlinear singular ordinary differential equation with irregular singularity. Our method of proof of this result is inspired by  Malgrange's version of Maillet theorem \cite{malgrange-maillet}. Existence of smooth Gevrey solution such as \cite{stolo-gevrey} or the existence of ``sectorial'' holomorphic solution such as in \cite{Ram-Mart1,Stolo-classif} are actually out of reach in that general context.

\medskip

Our claim that there are many significant local classification problems with big denominators
so that we can apply our main theorems to explain both some classical and some recently obtained results,
as well as to obtain new results on analytic normal forms, is confirmed in
section \ref{sec-appl} called ``Applications" and in the appendix section. In these sections we show what our main theorems
give for concrete local classification problems: of singular vector fields, of functions (including non-isolated singularities), of $n$-tuples of linearly independent vector fields on $\mathbb R^n$,
of Riemannian metrics, and of conformal structures.

\section{Big denominators in local classification problems}
\label{sec-big-denom-classif-problems}

\subsection{Classification problems with filtering action of a pseudo-group}

In order to describe a very wide class of local classification problems we will deal with, we need the following
notations:

\medskip

\noindent $\ank $ (resp. $\widehat{\ank}$) is the space of $k$-tuples of germs at $0\in \mathbb R^n$ (or $\Bbb C^n$) of analytic functions (resp. formal power series maps) of $n$ variables;

\medskip

\noindent $\left(\ank \right)^{(i)}$ is the homogeneous part of $\ank $ of degree $i\geq 0$;

\medskip

\noindent  $\left(\ank \right)_{>d}$ is the subspace of $\ank $ consisting of germs with zero $d$-jet at $0\in \mathbb R^n$;

We recall that 
$$
\widehat{\ank}=\bigoplus_{i\geq 0}\left(\ank \right)^{(i)}.
$$
\begin{definition}
The order at the origin of the formal power series $F=\sum_{i\geq 0}F^{(i)}\in \widehat{\ank}$ is the largest integer $k$, such that $F^{(i)}=0$, for all $i<k$ and $F^{(k)}\neq 0$. It will be denoted by $\text{ord}_0F$.
\end{definition}
\begin{definition}
\label{def-norms}
Let $i\ge 0$ and let $F = (F_1,..., F_k)\in \left(\mathbb A_n^k\right)^{(i)}$. Let
$F_j = \sum F_{j, \alpha }
x^\alpha $ where the sum is taken over all $j=1,...,k$ and all
multiindexes $\alpha = (\alpha _1,..., \alpha _n)$ such that $\vert \alpha \vert = \alpha _1+\cdots + \alpha _n = i$. Then
$$\vert \vert F_j \vert \vert = \sum _{\vert \alpha \vert = i}\vert F_{j,\alpha }\vert, \ \ \vert \vert F \vert \vert = max \left(\vert \vert F_1\vert \vert , \cdots ,  \vert \vert F_k\vert \vert \right).$$
\end{definition}
We recall that $F=\sum_{i\geq 0}F^{(i)}\in \widehat{\ank}$ defines a germ of analytic map if and  only if there exist, $M>0$ and $r>0$ such that $\|F^{(i)}\|\leq M (1/r)^i$. It is well known (see \cite{Grauert-L1} for instance) that the space
\begin{equation}\label{topo}
\cH_r:=\left\{F\in \widehat{\ank},\; \|F\|_r:=\sum_{i\geq 0}\|F^{(i)}\|r^i<+\infty\right\}
\end{equation}
is a Banach space. All these spaces define a basis of neighborhoods for $\Bbb A_n^k$.\\

Let us fix
\begin{equation}
\label{space}
P^{(q)} \in \left(\ans \right)^{(q)}, \ \ q\ge 0 \ \
\text{and the affine space} \ \
\mathcal A_{P^{(q)}} = P^{(q)} + \left(\ans \right)_{> q}.
\end{equation}
Let $r\in\Bbb N^*$ and let ${\bf m} = (m_1,\ldots,m_r)\in \Bbb N^r$ be a multiindex.
\medskip

\ni Let $\cG $ be a pseudo-group acting on $\mathcal A_{P^{(q)}}$.
We assume that $\cG$ has the form

\begin{equation}
\label{group}
\cG = id + {\mathcal F}_{r,\bf m}^{>0}, \ \ {\mathcal F}_{r,\bf m}^{>0} = \left(\mathbb A_n\right)_{> m_{1}} \times \left(\mathbb A_n\right)_{> m_{2}}
\times \cdots \times \left(\mathbb A_n\right)_{> m_{r}}
\end{equation}
The space ${\mathcal F}_{r,\bf m}^{>0}$ is filtered by the subspaces 
$$
{\mathcal F}_{r,\bf m}^{>i}:=\left(\mathbb A_n\right)_{> i+m_{1}} \times \left(\mathbb A_n\right)_{> i+m_{2}}
\times \cdots \times \left(\mathbb A_n\right)_{> i+m_{r}},\quad i\geq 0
$$
\begin{remark}
The infinite-dimensional vector space ${\mathcal F}$ parameterizes the {\it Lie algebra} of the group $\cG$.\
\end{remark}
\begin{definition}
We shall say that $F\in {\mathcal F}_{r,\bf m}^{>0}$ has order $>i$ at the origin if $F\in {\mathcal F}_{r,\bf m}^{>i}$ but $F\not\in {\mathcal F}_{r,\bf m}^{>i+1}$.
\end{definition}


\begin{problem}
\label{problem-construct-analytic}
To find as simple as possible  normal form $\mathcal N\subset \mathcal A_{P^{(q)}}$ serving for the whole $\mathcal A_{P^{(q)}}$, i.e. any $f\in \mathcal A_{P^{(q)}}$ is equivalent to some $\widetilde f\in \mathcal N$
with respect to the action of $\cG $.
\end{problem}

Certainly we need some properties of the action of $\cG $ . At first we assume that the action is
filtering which means the following. Denote
$${\mathcal F}^{(i)}_{r,\bf m} =  \left(\mathbb A_n\right)^{(m_{1}+i)} \times \left(\mathbb A_n\right)^{(m_{2}+i)}
\times \cdots \times \left(\mathbb A_n\right)^{(m_{r}+i)}, \ \ i\ge 1.
$$

\begin{definition}
\label{def-filtering}
The action of $\cG $ on $\mathcal A_{P^{(q)}}$ is filtering
 if for any $P^{(q)}+R\in \mathcal A_{P^{(q)}}$ and any $F\in {\mathcal F}_{r,\bf m}^{>0}$ one has
\begin{equation}
\label{eq-conjugacy}
(id + F)_*(P^{(q)}+R) = P^{(q)}+R + {\mathcal S}_{P^{(q)}}(F) + \mathcal T(R; F)
\end{equation}
where ${\mathcal S}_{P^{(q)}}$ is a {\it linear operator defined by $P^{(q)}$ only}, \ $\mathcal T(R,0) = 0$ and
\begin{equation}\label{eq-filtering1}
{\mathcal S}_{P^{(q)}}\left({\mathcal F}^{(i)}_{r,\bf m}\right)
\subseteq \left(\mathbb A_{n}^{s}\right)^{(q+i)},
\end{equation}
\begin{equation}\label{eq-filtering2}
\text{ord}_0\left(\mathcal T\left(R; F\right)- \mathcal T\left(R; G\right)\right))> \text{ord}_0\left(F- G\right)+q.
\end{equation}
\end{definition}

\begin{remark} 
The linear operator ${\mathcal S}_{P^{(q)}}$ is the linearization
of the action of $\cG$ at identity evaluated at $P^{(q)}$.
In the case where $\cG$ is the group of germs of diffeomorphisms at $0$, then ${\mathcal S}_{P^{(q)}}$ maps a germ of vector field $F$ to the Lie derivative of $F$ along $P^{(q)}$: ${\mathcal S}_{P^{(q)}}(F) =[\cS_{P^{(q)}}, F]= \frac{d}{dt} (\exp(tF)_*P^{(q)})_{\vert _{t=0}}$ where $\exp(tF)$ denotes the flow at time $t$ of the vector field $F$.
\end{remark}

\medskip

\noindent{\sc Notation}. By $\widehat{\mathcal G}=Id+\widehat{{\mathcal F}}$ where $\widehat{{\mathcal F}}=\left(\widehat{\mathbb A}_n\right)_{> d_{1}} \times \cdots \times \left(\widehat{\mathbb A}_n\right)_{> d_{r}}$ we will denote the group of formal transformations
corresponding to the group $\cG$ .

\subsection{Formal normal form}

The solution of Problem \ref{problem-construct-analytic} in the formal category is given by the following simple
statement.

\begin{prop}[Formal normal form]
\label{prop-formal-nf}
Assume that the action of $\cG$ on $\mathcal A_P$ is filtering.
Let $${\mathcal S}_{P^{(q)}}^{(i)}: \ {\mathcal F}^{(i)}_{r,\bf m} \ \to \left(\ans \right)^{(q+i)}$$ be the restriction of ${\mathcal S}_{P^{(q)}}$ to ${\mathcal F}^{(i)}$.
Fix any complementary subspaces $\mathcal N^{q+i}\subset \left(\ans \right)^{(q+i)}$
to the image of the operators ${\mathcal S}_{P}^{(i)}$:
\begin{equation}
\label{direct-sum}
\left(\ans \right)^{(q+i)} = Image \ {\mathcal S}_{P}^{(i)} \oplus \mathcal N^{q+i}, \ \ i\ge 1.
\end{equation}
Let $$\widehat{\mathcal N}= \mathcal N^{(q+1)} \oplus \mathcal N^{(q+2)}\oplus \cdots .$$
Then the affine space $P^{(q)} + \widehat{\mathcal N}$ is a formal normal form with respect to the action of $\cG$
serving for the whole affine space $\mathcal A_{P^{(q)}}$.
That is, for each $\widehat R\in \left(\widehat{\ans }\right)_{> q}$, there exists a formal transformation $\hat\Phi\in \hat {\mathcal G}$ such that
$$
{\hat\Phi}_*(P^{(q)}+R) -P^{(q)}\in \widehat{\mathcal N}.
$$
\end{prop}

\begin{proof}The filtering property allows to normalize the terms of order $q+1$, i.e. to bring them to $\mathcal N^{(q+1)}$,  by a transformation of form $id + \lambda $, $\lambda \in \mathcal F^{(1)}_{r,\bf m}$. Assume now that terms of order $\le q+p$ are normalized. The filtering property allows to normalize the terms of order $(q+p+1)$ by a
transformation of form $id + \lambda $, $\lambda \in \mathcal F^{(p+1)}_{r,\bf m}$ without changing the terms of order $\le q+p$. \end{proof}

\subsection{Theorem on analytic normal forms}

\begin{definition}
\label{def-formal-to-analytic}
 A formal normal form $P^{(q)} + \widehat{\mathcal N}$ holds in analytic category if for any $R\in \left(\ans \right)_{> q}$ there exists
$\Phi\in  {\mathcal G}$ such that the formal series of the analytic germ ${\hat\Phi}_*(P^{(q)}+R) -P^{(q)}$
belongs to $\widehat{\mathcal N}\cap \left(\ans \right)_{> q}$.
\end{definition}

Now Problem  \ref{problem-construct-analytic} can be specified as follows.

\begin{problem}
\label{problem-when-formal-is-analytic}
What has to be assumed in order to state that the normal form $P^{(q)} + \widehat{\mathcal N}$ holds in analytic category?
\end{problem}

\medskip

\ni We will give an answer involving the following definition.

\medskip

\ni {\bf Notation}. 
Given $F=(F_1,...,F_r)\in \mathcal F_{r,\bf m} $
and $x\in \mathbb R^n$ denote
$$j_x^{{\bf m}}F  =
\left(j^{m_{{1}}}_x F_1, \ \cdots , \ j^{m_{{r}}}_x F_r\right), \ \
J^{{\bf m}}\mathcal F_{r,\bf m} =
\left\{\left(x, \ j^{{\bf m}}_x F\right), \ \ x\in \mathbb R^n, \ F\in  \mathcal F_{r,\bf m}\right\}.$$

\begin{definition}
\label{def-diff-action}
The action of $\cG$ on $\mathcal A_{P^{(q)}}$ is an {\bf analytic differential action} of
order ${\bf m} = (m_1,...,m_r)$  if there exists a linear map $\cS$ depending only on $P^{(q)}$ and, for any $X\in \mathcal A_{P^{(q)}}$, there exists an analytic map germ $$W: \left(J^{{\bf m}}\mathcal F_{r,\bf m} , 0\right) \to \left(\mathbb R^s, 0\right)$$
such that, $W(x,0)=0$ and for any $x$ close to $0$ and any $F\in {\mathcal F}_{r,\bf m} $ with $j_x^{{\bf m}}F$ close to $0$,
$$
(id+F)_*X= X +\cS(F)+ W(x, j^{{\bf m}}_xF).
$$
\end{definition}
\begin{definition}\label{def-regular}
An analytic differential action of order ${\bf m} = (m_1,...,m_r)$ is said to be {\bf regular} if, for any formal map $F=(F_1,\ldots, F_r)$ with $\text{ord}_0F_i\geq m_i+1$, then
$$
\text{ord}_0\left(\frac{\partial W_i}{\partial u_{j,\alpha}}(x,\partial F)\right)\geq p_{j,|\alpha|},
$$
where
\begin{equation}\label{def-p}
p_{j,|\alpha|}=\max(0, |\alpha|+q+1-m_j).
\end{equation}
Here, we have set $\partial F:=\left(\frac{\partial^{|\alpha|} F_i}{\partial x^{\alpha}},\;1\leq i\leq r,\,0\leq |\alpha|\leq m_i  \right)$.
\end{definition}

\ni Given a formal normal form $P^{(q)} + \widehat{\mathcal N}$ denote by
$$\pi _N^{(q+i)}: \ \left(\ans \right)^{(q+i)} \ \to \text{Image } {\mathcal S}_{P^{(q)}}^{(i)}$$
the projection corresponding to the direct sum (\ref{direct-sum}).
Then the equation
\begin{equation}
\label{eq-with-pi}
{\mathcal S}_{P^{(q)}}^{(i)}(F^{(i)}) =
\pi _{N}^{(q+i)}(A^{(q+i)}), \ \ A^{(q+i)}\in \left(\ans \right)^{(q+i)}, \ \ i\ge 1
\end{equation}
has a solution
\begin{equation}
\label{solution-of-eq-with-pi}
F^{(i)} = \left( F^{(m_1+i)}_1,..., F^{(m_r+i)}_r \right)
\in  {\mathcal F}^{(i)}_{r,\bf m}
\end{equation}
for any $A^{(q+i)}\in \left(\ans \right)^{(q+i)}$.

\medskip

To formulate our main theorem on analytic normal forms we need to fix norms in the
spaces of homogeneous vector functions. 

\begin{thm}[Main theorem on analytic normal form]\label{main1}
Assume that the action of a pseudo-group $\cG $ of form
(\ref{group}) on an affine space of form (\ref{space}) is {\bf analytic differential} of order ${\bf m} = (m_1,\ldots, m_r)$, {\bf filtering} and {\bf regular}. Let $P^{(q)}+\widehat{\mathcal N}$ be a formal normal
form as constructed in Proposition \ref{prop-formal-nf}.
Assume there exists $C>0$ which depends neither on $i$ nor on
$A^{(q+i)}\in \left(\ans \right)^{(q+i)}$
such that equation (\ref{eq-with-pi}) has a solution (\ref{solution-of-eq-with-pi})
satisfying the estimates
\begin{equation}
\label{estimates-big-denom}
\vert \vert F_{1}^{(m_1+i)}\vert \vert < C\frac{\vert \vert A^{(q+i)}\vert \vert }{i^{m_{1}}}, \ \
\cdots , \ \ \vert \vert F_{r}^{(m_r+i)}\vert \vert < C\frac{\vert \vert A^{(q+i)}\vert \vert }{i^{m_{r}}}.
\end{equation}
The formal normal form $P^{(q)}+\widehat{\mathcal N}$ holds in analytic category.
\end{thm}

\begin{remark}
In section \ref{sec-appl} and in the Appendix, we present several natural local classification problems for which the assumptions of Theorem 2.13 hold 
true.
\end{remark}

The main assumption $(\ref{estimates-big-denom})$ of the previous Theorem means that we have the following two properties:

\medskip


\noindent $\bullet$ {\it The big denominators property.} It is a property of
the linear operators $\mathcal S_{P^{(q)}}^{(i)}$ and it has nothing to do with
the formal normal form. This property is as follows:

\medskip

\noindent there exists $C>0$ such that for any $i\ge 1$ and any
 $A^{(q+i)}\in \text{Image } \mathcal S_{P^{(q)}}^{(i)}$
 the equation ${\mathcal S}_{P^{(q)}}^{(i)}(F^{(i)}) =
A^{(q+i)}$ has a solution (\ref{solution-of-eq-with-pi})
 satisfying (\ref{estimates-big-denom}).

\medskip

\noindent $\bullet $ {\it Uniformly bounded formal normal form}.
This property concerns the choice of the complementary spaces in Proposition \ref{prop-formal-nf} which
defines the formal normal form $P^{(q)}+\widehat{\mathcal N}$. We will say that this formal normal form is uniformly bounded if for any $i\ge 1$ and any $A^{(q+i)}\in \left(\ans \right)^{(q+i)}$ one has
 $ \vert \vert \pi _{N}^{(q+i)}(A^{(q+i)}) \vert \vert < C\vert \vert A^{(q+i)}\vert \vert $ for some constant $C>0$ which depends neither on $i$ nor on
 $A^{(q+i)}\in \left(\ans \right)^{(q+i)}$.

\medskip

So, the assumption $(\ref{estimates-big-denom})$ of Theorem \ref{main1} is equivalent to the assumption that
the operator $\mathcal S_{P^{(q)}}$ has big denominator property along with the assumption that the formal normal form
$P^{(q)}+\widehat{\mathcal N}$ is uniformly bounded.

\begin{remark}
One may consider a subspace ${\cal B}_{P^{(q)}}$ of ${\cal A}_{P^{(q)}}$ of higher order perturbations of $P^{{(q)}}$ that preserve a property. It could be, for instance, commuting with an involution or leaving invariant a differential form. Then, we will use the subgroup $\tilde {\cal G}$ of ${\cal G}$ that leave the property invariant. Then, one has to consider a subspace $\tilde {\cal F}^{>0}_{r,m}$ of ${\cal F}^{>0}_{r,m}$, namely the ``Lie algebra'' of $\tilde{\cal G}$. The operator ${\cal S}$ has to be considered as a map from $\tilde {\cal F}^{>0}_{r,m}$ to ${\cal B}_{P^{(q)}}$ with the induced topology. We refer to \cite{stolo-lombardi}[section 4.3] for a similar discussion in the case of vector fields (the map ${\cal S}$ is denoted there $d_0$).
\end{remark}

\begin{remark}
The theorem is also true if instead using the norms of definition \ref{def-norms}, we use a norm that satisfies the following requirements~: 
\begin{enumerate}
\item $f\in \mathbb A_n$ if and only if $\widehat f (t):= \sum_{i\geq 0}\|f^{(i)}\|t^i$ is analytic at $0\in \Bbb C$, where $f^{(i)}$ denotes the homogeneous part of degree $i$ of the formal power series $f$.
\item $\|\bar f^{(i)}\|=\|f^{(i)}\|$ ($\bar f^{(i)}$ is the polynomial obtained by replacing the coefficients of $f^{(i)}$ by their absolute value)
\item 
$$
\widehat{\frac{\partial^Q f}{\partial x^{|Q|}}}\prec \frac{\partial^{|Q|}\widehat{f}}{\partial t^{|Q|}}.
$$
\item
$$
\widehat{fg} \prec \widehat{f}.\widehat{g}.
$$
\end{enumerate}
Besides, characterization of convergent power series, these requirements are those of lemma \ref{lem-deriv}.
%
For instance, the modified Belitskii scalar product defined as follow~:
$$
\left<f,g\right>_{MB} = \sum _{\vert \alpha \vert = i}\frac{\alpha !}{|\al|!}f_\alpha \bar g_\alpha
$$
satisfied these requirements (as shown in \cite{stolo-lombardi} and remark \ref{belit-lemma}).
\end{remark}

\section{Big denominators in non-linear systems of PDEs}
\label{sec-PDEs}

In this section we generalize Theorem \ref{main1} to a theorem on local analytic solvability of non-linear systems of
PDEs. We define big denominators for such systems and prove the local analytic solvability in the case of big denominators. We explain that Theorem \ref{main1} is a simple corollary of Theorem \ref{main-thm} in this section.
It means that a ``right place" of Theorem \ref{main1} is the theory of non-linear PDEs rather than
local classification problems. Nevertheless the applications of Theorem \ref{main-thm} that we know concern
namely local classification problems. We refer to \cite{treves-Ovcyannikov} for analytic theory of PDE's. For the study of some singularities analytic of PDE's, we refer to \cite{baouendi-goulaouic-fuchs,gerard-tahara-book}. Our references for singularity analytic theory of ODE's are \cite{wasow, sibuya-textbook}. 

\subsection{The problem}

Let ${\bf m} = (m_1,\ldots,m_r)\in \Bbb N^r$ be a fixed multiindex. Given $F=(F_1,...,F_r)\in \mathcal F_{r,\bf m}^{>0}$
and $x\in \mathbb R^n$ denote 
$$
j_x^{{\bf m}}F  =\left(j^{m_{{1}}}_x F_1, \ \cdots , \ j^{m_{{r}}}_x F_r\right), \ \
J^{{\bf m}}\mathcal F_{r,\bf m}^{>0} =\left\{\left(x, \ j^{{\bf m}}_x F\right), \ \ x\in (\mathbb R^n,0), \ F\in  \mathcal F_{r,\bf m}^{>0}\right\}.
$$
%
\begin{definition}
\label{def-diff-map}
A map $\mathcal T: \cF_{r,\bf m}^{>0}\to \mathbb A_n^s$ is a {\bf differential analytic map of order ${\bf m}$} at the point
$0\in \mathbb A_n^k$ if there exists an analytic map germ $$W: \left(J^{{\bf m}}\mathcal F_{r,\bf m}^{>0}, 0\right) \to \left(\mathbb R^s, 0\right)$$
such that $\mathcal T(F)(x) = W(x, j^m_xF)$ for any $x\in \mathbb R^n$ close to $0$ and any function germ $F\in \cF_{r,\bf m}^{>0}$ such that $j^m_0F$ is close to $0$.
\end{definition}

Denote by $$v = \left(x_1,...,x_n, u_{j, \alpha }\right), \ \
  1\leq j\leq r, \ \alpha = (\alpha _1,..., \alpha _n)\in\Bbb N^n, \ \vert \alpha \vert \le m_j$$
   the local coordinates in $J^m\mathbb A_n^r$, where $u_{j, \alpha }$ corresponds to the partial
   derivative $\partial^{|\alpha|} /\partial x_1^{\alpha _1}\cdots \partial x_n^{\alpha _n}$ of the $j$-th component
   of a vector function $F\in \mathbb A_n^r$. 
%

\begin{definition}
Let $\mathcal T: \cF_{r,\bf m}^{>0}\to \mathbb A_n^s$ be a map. 
\begin{itemize} 
\item We shall say that it {\bf increases the order at the origin} (resp. strictly) by $q$ if  
for all $(F,G)\in (\cF_{r,\bf m}^{>0})^2$ then 
$$
\text{ord}_0\left({\mathcal T}(F)-{\mathcal T}(G)\right)\geq \text{ord}_0(F-G)+q,
$$
(resp. $>$ instead of $\geq$).
\item Assume that ${\cT}$ is an analytic differential map of order $\bf m$
defined by a map germ $W:  \left(J^{{\bf m}}\mathcal F_{r,\bf m}^{>0}, 0\right) \to \left(\mathbb R^s, 0\right)$ as in Definition \ref{def-diff-map}. 
We shall say that it is {\bf regular} if, for any formal map $F=(F_1,\ldots, F_r)\in \widehat\cF_{r,\bf m}^{>0}$, then
$$
\text{ord}_0\left(\frac{\partial W_i}{\partial u_{j,\alpha}}(x,\partial F)\right)\geq p_{j,|\alpha|},
$$
where $p_{j,|\alpha|}=\max(0, |\alpha|+q+1-m_j)$. As above, we have set $\partial F:=\left(\frac{\partial^{|\alpha|} F_i}{\partial x^{\alpha}},\;1\leq i\leq r,\,0\leq |\alpha|\leq m_i  \right)$.
\end{itemize}
\end{definition}

Let us consider linear maps\footnote{see remark \ref{pi-formal}}~:
\begin{enumerate}
\label{triple}
\item 
$$
\mathcal S: \cF_{r,\bf m}^{>0}\ \to \ \mathbb A_n^s,
$$
that increases the order by $q$.
\item 
$$
\pi : \mathbb A_n^s \to Image \hskip .05cm (\mathcal S) \subset \mathbb A_n^s
$$
is projection onto $ Image \hskip .05cm (\mathcal S)$.
\end{enumerate}

Let us consider a differential analytic map of order $\bf m$, $\mathcal T:\cF_{r,\bf m}^{>0}\ \to \ \mathbb A_n^s$. 

We consider the equation
\begin{equation}
\label{eq-main}
\mathcal S(F) = \pi \left( \mathcal T(F)\right)
\end{equation}
The problem is to find a sufficient condition on the triple $\left(\mathcal S, \mathcal T, \pi \right)$ under which
equation (\ref{eq-main}) has a solution $F\in \cF_{r,\bf m}^{>0}$.
In what follows we will prove that one of sufficient conditions is the ``big denominators property"
of the triple $(\mathcal S,\mathcal T,\pi )$ defined in this section below.

\subsection{Big denominators. Main theorem}

\begin{definition}
Let $k\in\Bbb N^*$. Given $F\in \mathbb A_n^k$ we define an analytic function germ of one complex variable $z\in \mathbb C$ by 
$$\widehat F(z) := \sum _{i\ge 0} \vert \vert F^{(i)}\vert \vert z^i, \ \ \ \ z\in \mathbb C,$$
where $F^{(i)}$ denotes the homogeneous degree $i$ part of $F$ of the Taylor expansion at the origin. The norms in the spaces of homogeneous vector functions are those of Definition \ref{def-norms}.
\end{definition}

\begin{definition}\label{domination}
Given two formal power series $F=\sum F_\alpha z^\alpha $ and $G=\sum G_\alpha z^\alpha $ of $n$ complex variables
 $z = (z_1,...,z_n)$ we will say that $G$ dominates $F$ ($F\prec G)$ if
 $$  G_\alpha \geq 0 \ \text{and}\  |F_\alpha |\leq G_\alpha  \ \
\text{ for all multi-indexes}\  \alpha = (\alpha _1,...,\alpha _n).$$
We also denote
$$\overline F = \sum \vert F_\alpha \vert x^\alpha .$$
\end{definition}

Now we can define the big denominators property of the triple $(\mathcal S, \mathcal T, \pi )$ in equation (\ref{eq-main}).

\begin{definition}
\label{def-big-denominators}
The triple of maps $(\mathcal S, \mathcal T, \pi )$ of form (\ref{triple})
has {\bf big denominators property of order $\bf m$}
if there exists an nonnegative integer $q$ such that the following holds:

\begin{enumerate}

\item \ $\mathcal T$ is an regular analytic differential map of order $\bf m$ that strictly increases the order by $q$ and $j^{q}_0\mathcal T(0)=0$.

\item $\cS: \cF_{r,\bf m}^{>0}\ \to \ \mathbb A_n^s$ is linear and increases the order by $q$.

\item the linear map $\pi : \mathbb A_n^s \to Image \hskip .05cm (\mathcal S) \subset \mathbb A_n^s$ is a projection.

\item \ the map $\mathcal S$ admits right-inverse $\mathcal S^{-1}: Image (S)\to \mathbb A_n^r$ such that
the composition $\mathcal S^{-1}\circ \pi $ satisfies:

there exists $C>0$ such that for any $G\in \mathbb A_n^s$ of order $\geq q+1$, one has for all $1\leq i\leq r$
\begin{equation}
\label{BD}
\frac{d^{m_i} z^{q}\widehat{\cS_i^{-1}\circ \pi (G)}}{dz^{m_i}} \prec \ C\widehat{G}.
\end{equation}
where ${\cal S}_i^{-1}$ denotes the $i$th component of ${\cal S}^{-1}$, $1\leq i\leq r$.
\end{enumerate}
\end{definition}

Our main theorem is as follows.

\begin{thm}[Main theorem on non-linear systems of PDEs]
\label{main-thm}
Let us consider a system of analytic non-linear pde's such as equation (\ref{eq-main})~:
$$
\mathcal S(F) = \pi \left( \mathcal T(F)\right).
$$
If the triple $(\mathcal S, \mathcal T,\pi )$  has big denominators property of order $\bf m$,
according to definition \ref{def-big-denominators}, then the equation has an analytic solution $F\in \cF_{r,\bf m}^{>0}$.
\end{thm}

\begin{remark}\label{rem-top-bd}
Condition $(\ref{BD})$ means that, for all $i\geq 1$, 
\begin{equation}\label{BD-expand}
\left\|\left(\cS_j^{-1}\circ \pi (G)\right)^{(i+m_j)}\right\|\leq C\frac{\|G^{(i+q)}\|}{(i+m_j+q)\cdots (i+q+1)}.
\end{equation}
Hence, if $G\in \mathbb A_n^s$ if of order $\geq q+1$, and $r>0$ then 
$$
\vert \vert \mathcal S_j^{-1}\circ \pi (G)\vert \vert_r = \sum_{i\geq 1}\vert \vert \left(\mathcal S_i^{-1}\circ \pi (G)\right)^{(i+m_j)}\vert \vert r^{m_j+i}\leq C \sum_{i\geq 1} \vert \vert G^{(i+q)}\vert \vert r^{i+m_j} = Cr^{m_j-q}\|G\|_r.
$$
Hence, the maps $\mathcal S_j^{-1}\circ \pi$, $1\leq j\leq r$, are continuous with respect to the topology defined by $\cH_r$ (see $(\ref{topo})$).
\end{remark}
\begin{remark}\label{pi-formal}
The linear maps $\cS$ and $\pi$ do not need to be continuous with respect to the topology induced by ${\cH}_r$. In fact they can to be thought just as formal maps $\widehat{\cS}: \widehat{\cF_{r,\bf m}^{>0}}\ \to \ \widehat{\mathbb A_n^s}$, $\pi : \widehat{\mathbb A_n^s}\to Image \hskip .05cm (\widehat{\mathcal S}) \subset \widehat{\mathbb A_n^s}$ ($\widehat{\cF_{r,\bf m}^{>0}}$ denotes the formal completion of $\cF_{r,\bf m}^{>0}$). As mentioned above, the big denominators property leads to the fact $\mathcal S^{-1}\circ \pi$ is continuous with respect to that topology.
\end{remark}

\begin{thm}[Variation of the main theorem on non-linear systems of PDEs]
\label{var-main-thm}
Let ${\bf m} = (m_1,\ldots,m_r)\in \Bbb N^r$ be a fixed multiindex. Let $q$ be a nonnegative integer and let $d\geq -\min_i m_i$ be an integer. Let us consider a triple $(\mathcal S, \mathcal T,\pi )$  such that 
\begin{enumerate}

\item \ $\mathcal T$ is an analytic differential map of order $\bf m$ that strictly increases the order by $q-d$ and $j^{q}_0\mathcal T(0)=0$.
\item $\cT$ is  {\bf $d$-regular} that is~: for any formal map $F=(F_1,\ldots, F_r)$ with $\text{ord}_0F_i\geq m_i+1+d$, then
$$
\text{ord}_0\left(\frac{\partial W_i}{\partial u_{j,\alpha}}(x,\partial F)\right)\geq p_{j,|\alpha|},
$$
where
\begin{equation}\label{d-def-p}
p_{j,|\alpha|}=\max(0, |\alpha|+q+1-m_j-d).
\end{equation}
Here, we have set $\partial F:=\left(\frac{\partial^{|\alpha|} F_i}{\partial x^{\alpha}},\;1\leq i\leq r,\,0\leq |\alpha|\leq m_i  \right)$.
\item $\cS: \cF_{r,\bf m}^{>d}\ \to \ \mathbb A_n^s$ is linear and increases the order by $q-d$.

\item the linear map $\pi : \mathbb A_n^s \to Image \hskip .05cm (\mathcal S) \subset \mathbb A_n^s$ is a projection.

\item \ the map $\mathcal S$ admits left-inverse $\mathcal S^{-1}: Image (S)\to \mathbb A_n^r$ such that
the composition $\mathcal S^{-1}\circ \pi $ satisfies:


there exists $C>0$ such that for any $G\in \mathbb A_n^s$ of order $\geq q+1$, one has for all $1\leq i\leq r$
\begin{equation}
\label{d-BD}
\frac{d^{m_i} z^{q-d}\widehat{\cS_i^{-1}\circ \pi (G)}}{dz^{m_i}} \prec \ C\widehat{G}.
\end{equation}
\end{enumerate}
Then the equation ~:
$$
\mathcal S(F) = \pi \left( \mathcal T(F)\right).
$$
has an analytic solution $F\in \cF_{r,\bf m}^{>d}$.
\end{thm}
\subsection{From theorem \ref{main-thm} to theorem \ref{main1}}

Let us prove that theorem \ref{main1} is a corollary of theorem \ref{main-thm}. 
We are given an analytic differential action of order ${\bf m}$. Hence, given $R_{>q}\in \Bbb A_n^s$ of order $>q$ and $F\in \cF_{r,\bf m}^{>0}$, the action is defined, according to $(\ref{eq-conjugacy})$, by
$$
(id+F)_*(P^{(q)}+R_{>q})= P^{(q)}+R_{>q} +\cS(F)+ \cT(R_{>q}, F).
$$
and $\cT(R_{>q}, F)= W(x, j^{{\bf m}}_xF)$ for some analytic map $W$ at the origin in the jet-space.
We assume that $\mathcal S$ and $\pi $ are linear operators such that, one has
\begin{equation*}
\mathcal S\left(\cF_{r,\bf m}^{(i)}\right) \subset
(\mathbb A_n^s)^{(q+i)}, \ \ \pi \left((\mathbb A_n^s)^{(i)}\right) \subset (\mathbb A_n^s)^{(i)} \cap \ Image (\mathcal S)
\end{equation*}
The projection operator $Id-\pi$ is supposed to define the formal normal form space. Therefore, that $(Id+F)$ conjugates $P^{(q)}+R_{>q}$ to a normal form means that
$$
\pi\left((id+F)_*(P^{(q)}+R_{>q})- P^{(q)}\right)=0.
$$
In other words, $F$ is solution to the problem
$$
\cS(F)+ \pi\left(R_{>q} +W(x, j^{{\bf m}}_xF)\right)=0.
$$
Let us set $\cT(F):=-(R_{>q} +W(x, j^{{\bf m}}_xF))$. According the the filtering property $(\ref{eq-filtering1})$ (resp. $(\ref{eq-filtering2})$), $\cS$ (resp. $\cT$ strictly) increases the order by $q$. Moreover, $\cT$ is regular since the action is, so that $W$ is regular. Moreover, $\cT(0)=-R_{>q}$ has order $>q$ at the origin. Hence, the first three points of definition \ref{def-big-denominators} are satisfied.

The operator ${\mathcal S}$ can have a kernel. So each space $\cF_{r,\bf m}^{(i)}$ of ``homogeneous'' polynomial mapping degree $i$, can be decomposed as a direct sum into
$$
\cF_{r,\bf m}^{(i)}= \ker {\mathcal S}^{(i)}_{|\cF_{r,\bf m}^{(i)}}\oplus C^{(i)}.
$$
for some {\bf chosen} subspace $C^{(i)}$. This choice defines a right inverse of ${\mathcal S}^{(i)}$~: 
$$({\mathcal S}^{(i)})^{-1}:\text{Image } {\mathcal S}^{(i)}\rightarrow C^{(i)}.$$
And ${\mathcal S}^{(i)}$ is injective on $C^{(i)}$. 

Estimates $(\ref{estimates-big-denom})$ can rephrased as follows~: for any any $i\ge 1$ and $A^{(i+q)}\in (\mathbb A_n^s)^{(i+q)}$, we have
$$
\left\|\left(\mathcal S_j^{-1}\circ \pi (A^{(i+q)})\right)^{(i+m_j)}\right\| \le C\frac {\vert \vert A^{(i+q)}\vert \vert }{i^{m_j}},\quad j=1,\ldots, r \ \ 
$$
with a constant $C>0$ which does not depend on $i$ and on $G$. Hence, $(\ref{BD})$ is satisfied.

\section{Proof of the main theorem \ref{main-thm}}
\label{sec-proofs}

The proof goes as follow~: We first prove to that there exists a unique formal solution to the problem. Then, we show that this formal solution is dominated by a solution of an analytic system of nonlinear ordinary differential equations that has a a kind of {\it regular singularity} at the origin. We then prove that this last solution is indeed analytic in a neighborhood of the origin.

\subsection{Existence of a formal solution}

Given equation (\ref{eq-main}), with $(\mathcal S, \mathcal T, \pi)$ having the big denominators property of order $\bf m$, let us first show that there exists a {\bf formal solution} $F$ to 
$$F = \mathcal S^{-1}(\pi ( \mathcal T(F))), $$

We look for an $r$-tuple of formal power series $F=\sum_{k\geq 1} F^{(k)}$ where $F^{(k)}:=(F_1^{(m_1+k)}, \ldots, F_r^{(m_r+k)})$, $F_j^{(m_j+k)}$ is an homogeneous polynomial of degree $m_j+k$. We shall say such a r-tuple $F^{(k)}$ is ``homogeneous of degree $k$'' when it does not lead to confusion. The order of an $r$-tuple of formal power series is the greatest integer $i$ such that $F^{(i)}\neq 0$.

Let us prove by induction on the degree $k\geq 1$, that $F^{(k)}$ is uniquely determined by
$$
F^{(1)}:= \left({\mathcal S}^{-1}\circ\pi({\mathcal T}(0))\right)^{(1)},\quad F^{(k)}:= \left({\mathcal S}^{-1}\circ\pi({\mathcal T}(F_{m+1}+\cdots+F_{k-1}))\right)^{(k)},\; k>1.
$$ 
By assumption, we have $\text{ord}_0\mathcal T(0)\geq q+1$. Since $\mathcal T$ strictly increases the order by $q$, then 
$$
\text{ord}_0({\mathcal T}(F)-{\mathcal T}(0))>q+1,\;\text{i.e.}\;j_0^{q+1}({\mathcal T}(F)-{\mathcal T}(0))=0.
$$
According to the big denominators properties, we have
$$
\text{ord}_0\left(\left(\mathcal S^{-1}(\pi ( \mathcal T( F)))- \mathcal S^{-1}(\pi ( \mathcal T(0)))\right)\right)>1.
$$
On the other hand, we have
$$
F= {\mathcal S}^{-1}\circ\pi({\mathcal T}(0))+ \left(\mathcal S^{-1}(\pi ( \mathcal T( F)))- \mathcal S^{-1}(\pi ( \mathcal T(0)))\right).
$$
Therefore, $F^{(1)}$ is uniquely defined by
$$
F^{(1)}:= \left(\mathcal S^{-1}(\pi ( \mathcal T(0)))\right)^{(1)}.
$$
Assume that $F^{(1)},\ldots, F^{(k-1)}$ are known by induction. Let us set $S_{k-1}:= \sum_{i=1}^{k-1}F^{(i)}$. Let us show that $F^{(k)}=\left({\mathcal S}^{-1}\circ\pi({\mathcal T}(S_{k-1}))\right)^{(k)}$. Indeed, we have 
$$
F= {\mathcal S}^{-1}\circ\pi({\mathcal T}(S_{k-1}))+ \left(\mathcal S^{-1}(\pi ( \mathcal T( F)))- \mathcal S^{-1}(\pi ( \mathcal T(S_{k-1})))\right)
$$
We have 
$$
\text{ord}_0\left(\mathcal S^{-1}(\pi ( \mathcal T( F)))- \mathcal S^{-1}(\pi ( \mathcal T(S_{k-1})))\right)\geq \text{ord}_0\left( \mathcal T(F)- {\mathcal T(S_{k-1}})\right)-q
$$
Since ${\mathcal T}$ increases the order by $q$, we have 
$$
\text{ord}_0\left( \mathcal T(F)- {\mathcal T(S_{k-1}})\right)>\text{ord}_0\left( F-S_{k-1}\right)+q.
$$
Hence, we have 
$$
\text{ord}_0\left(\mathcal S^{-1}(\pi ( \mathcal T( F)))- \mathcal S^{-1}(\pi ( \mathcal T(S_{k-1})))\right)> k
$$
and we are done.

\subsection{Existence of an analytic majorant}

In the previous section, we proved that the problem has a formal solution. We shall prove that this solution is dominated by a germ of analytic map. This germ is defined to the solution of some system of analytic nonlinear ordinary differential equations.
The key statement in the proof of Theorem \ref{main-thm} is as follows.
\begin{prop}
\label{prop-to-construct-G}
Under assumptions of theorem \ref{main-thm}, there exists a nonnegative integer $k$, a germ $G(t,u_{j,l},\;1\leq i\leq r,\; 0\leq l\leq m_j)$ of analytic function of $(m_1+1)+\cdots +(m_r+1)+1$ variables and there exists germs of analytic functions $f_1,\ldots, f_r$ of one variable vanishing at the origin such that, for all $1\leq j\leq r$,
\begin{enumerate}
\item 
\begin{equation}\label{diff-equ}
\frac{d^{m_j}z^{m_j+q+k}f_j}{dz^{m_j}} (z) = CG\left(z,\frac{d^lz^{m_j+k}f_j(z)}{dz^l},\;1\leq j\leq r,\; 0\leq l\leq m_j\right). 
\end{equation}
Here, $C$ is the constant involved in the big denominators property. 
\item 
\begin{equation}\label{domination}
\widehat{F_j(z)-F_j^{\leq m_j+k}}  \prec   z^{m_j+k}f_i(z),\quad j=1,\ldots, r.
\end{equation}
\end{enumerate}
\end{prop}

We shall first construct the majorant function $G$ and we shall then prove that if the $f_j$'s are formal solutions of the system $(\ref{diff-equ})$, then the estimates $(\ref{domination})$ hold. On the other hand, we shall prove that the system $(\ref{diff-equ})$ has indeed a germ of analytic solution vanishing at the origin. The regularity assumption will say that the system has only a regular singularity at the origin.

\subsubsection{Construction of majorant system of differential equations with regular singularity}\label{proof-main}

The function $G$ of $m+2$ variables in Proposition \ref{prop-to-construct-G} can be constructed as follows,
from the map $W: \ J^m\mathbb A_n^r \to \mathbb R^s$ in Definition \ref{def-diff-map} and the constant $C$
in item 2.2 of Definition \ref{def-big-denominators}. Denote by $$v = \left(x_1,...,x_n, u_{j, \alpha }\right), \ \
  j\in \{1,...,k\}, \ \alpha = (\alpha _1,..., \alpha _n), \ \vert \alpha \vert \le m$$
   the local coordinates in $J^m\mathbb A_n^r$, where $u_{j, \alpha }$ corresponds to the partial
   derivative $\partial^{|\alpha|} /\partial x_1^{\alpha _1}\cdots \partial x_n^{\alpha _n}$ of the $j$-th component
   of a vector function $F\in \mathbb A_n^r$. 

The mapping $W = (W_1,...,W_s)$ is analytic in neighborhood of $0$ and $W(0)=0$. We recall that the equation has a unique formal solution $F=(F_1,\ldots, F_r)$. 
Let $k\geq 1$ and let us define
$$
F^{\leq k}:=(F_1^{\leq k+m_1},\ldots , F_r^{\leq k+m_r}).
$$
In other words, $F^{\leq k}$ is the vector which coordinates are the $(k+m_j)$-jet a the origin of $F_j$, $1\leq j\leq r$.

In order to simplify the notation, we shall also write
$$
\partial F^{\leq k}:=\left(\frac{\partial^{|\alpha|} F_i^{\leq m_i+k}}{\partial x^{\alpha}}, 1\leq i\leq r,\; |\alpha|\leq m_i\right)
$$

Let us write the formal solution as a sum $F:=F^{\leq k}+F^{>k}$ where $F^{>k}=(F_1^{>m_1+k},\ldots, F_r^{>m_r+k})$ and $\text{ord}_0(F_i^{>m_i+k})> m_i+k$. Let us Taylor expand $\cT(F)=W(x,\partial F)$ at the point $(x,\partial F^{\leq k})$ at order 2 with integral rest. We have
\begin{eqnarray*}
W_i\left(x,\partial F\right)& =& W_i(x,\partial F^{\leq k})+\sum_{j=1}^r\sum_{|\alpha|\leq m_j}\frac{\partial W_i}{\partial u_{j,\alpha}}(x,\partial F^{\leq k})\frac{\partial^{|\alpha|} F_j^{>m_j+k}}{\partial x^{\alpha}}\\
& &+ \sum_{j,j'=1}^r\sum_{|\alpha|,|\alpha'|\leq m_j}H_{i,j,j',\alpha,\alpha'}(x,\partial F^{\leq k}, \partial F^{>k} )\frac{\partial^{|\alpha|} F_j^{>m_j+k}}{\partial x^{\alpha}}\frac{\partial^{|\alpha'|} F_{j'}^{>m_j'+k}}{\partial x^{\alpha'}}.
\end{eqnarray*}
Here, $H_{i,j,j',\alpha,\alpha'}$ is analytic in all its variables at the origin. The equation to solve is $\cS(F) = \pi(W(x,\partial F))$.
{\bf We assume $\cS$ to be linear and increases the order by $q$. We also assume that $\pi$ is a linear projection}. This leads to 
\begin{equation*}
F^{>k}=\cS^{-1}\circ\pi\left(W\left(x,\partial F\right)\right)-F^{\leq k}.
\end{equation*}
Since $F^{>k}$ is of order $>k$ at the origin, then 
\begin{equation}\label{equation-k}
\cS(F^{>k})=\pi(W\left(x,\partial F\right))-\cS(F^{\leq k})=\pi\left(W\left(x,\partial F\right)-\cS(F^{\leq k})\right)
\end{equation}
has order $\geq k+q+1$ at the origin.
Since $\cS$ (resp. $\cT$) increases (resp. strictly) the order by $q$, we have
\begin{eqnarray*}
\text{ord}_0\left(\cS(F^{\leq k})-\cS(F)\right)&\geq &\text{ord}_0\left(F^{\leq k}-F\right)+q\geq k+1+q\\
\text{ord}_0\left(\cT(F^{\leq k})-\cT(F)\right)& > &\text{ord}_0\left(F^{\leq k}-F\right)+q> k+1+q.
\end{eqnarray*}
Therefore, we have 
\begin{equation}\label{ini-equ}
\text{ord}_0\left(\cS(F^{\leq k})-\cT( F^{\leq k})\right)\geq k+1+q.
\end{equation}

Therefore, we can majorize as follow~:
\begin{eqnarray}
\overline{W_i\left(x,\partial F\right)-\cS_i(F^{\leq k})} & =& w_{i,k}+\sum_{j=1}^r\sum_{|\alpha|\leq m_j}\overline{\frac{\partial W_i}{\partial u_{j,\alpha}}}(x,\partial \bar F^{\leq k})\frac{\partial^{|\alpha|} \bar F_j^{>m_j+k}}{\partial x^{\alpha}}\\
& &+ \sum_{j,j'=1}^r\sum_{|\alpha|,|\alpha'|\leq m_j}\bar H_{i,j,j',\alpha,\alpha'}(x,\partial \bar F^{\leq k}, \partial \bar F^{>k} )\frac{\partial^{|\alpha|} \bar F_j^{>m_j+k}}{\partial x^{\alpha}}\frac{\partial^{|\alpha'|} \bar F_{j'}^{>m_{j'}+k}}{\partial x^{\alpha'}}.\nonumber
\end{eqnarray}
where we have set $w_{i,k}:=\overline{W_i(x,\partial F^{\leq k})-\cS_i(F^{\leq k})}$.

We recall that 
\begin{equation}\label{def-p}
p_{j,|\alpha|}=\max(0, |\alpha|+q+1-m_j).
\end{equation}
According to $\ref{ini-equ}$ and by assumption, we have
\begin{eqnarray*}
\text{ord}_0(w_{i,k})&\geq &k+q+1\\
\text{ord}_0\left(\frac{\partial W_i}{\partial u_{j,\alpha}}(x,\partial F^{\leq k})\right)&\geq &p_{j,|\alpha|}\\
\end{eqnarray*}

Let us fix the inter $k>0$.
Since each function $H_{i,j,j',\alpha,\alpha'}(x,\partial F^{\leq k}, u_{p,\beta}, 1\leq p\leq r, 0 \leq |\beta|\leq m_p )$ is analytic in a neighborhood of the origin in $\Bbb R^t$, then there exists positive constants $M, c$ such that
$$
H_{i,j,j',\alpha,\alpha'}(x,\partial F^{\leq k}, u_{p,\beta}, 1\leq p\leq r, 0 \leq |\beta|\leq m_p )\prec\frac{M}{1-c\left(x_1+\ldots + x_n+\sum_{p=1}^r\sum_{k=0}^{m_r}\sum_{|\alpha|=k}u_{p,\alpha}\right)}.
$$
Furthermore, each function $w_{i,k}$ and $\frac{\partial W_i}{\partial u_{j,\alpha}}(x,\partial F^{\leq k})$ is analytic in same neighborhood of the origin in $\Bbb R^n$. Hence, we have
\begin{eqnarray*}
w_{i,k}(x)&\prec &\frac{M(x_1+\cdots x_n)^{k+q+1}}{1-c(x_1+\cdots x_n)},\\
\frac{\partial W_i}{\partial u_{j,\alpha}}(x,\partial F^{\leq k})&\prec& \frac{M(x_1+\cdots x_n)^{p_{j,|\alpha|}}}{1-c(x_1+\cdots x_n)}
\end{eqnarray*}
As a consequence, we have
\begin{eqnarray}\label{majorant}
\overline{W_i\left(x,\partial F\right)-\cS_i(F^{\leq k})} & \prec & \frac{M}{1-c(x_1+\cdots x_n)}\left[ (x_1+\cdots x_n)^{k+q+1}+\right.\\
&&\left.\sum_{j=1}^r\sum_{|\alpha|\leq m_j}(x_1+\cdots x_n)^{p_{j,|\alpha|}}\frac{\partial^{|\alpha|} \bar F_j^{>m_j+k}}{\partial x^{\alpha}}\right]\nonumber\\
& &+\frac{M\left( \sum_{j,j'=1}^r\sum_{|\alpha|,|\alpha'|\leq m_j}\frac{\partial^{|\alpha|} \bar F_j^{>m_j+k}}{\partial x^{\alpha}}\frac{\partial^{|\alpha'|} \bar F_{j'}^{>m_{j'}+k}}{\partial x^{\alpha'}}\right)}{1-c\left(x_1+\ldots +x_n+\sum_{p=1}^r\sum_{k=0}^{m_r}\sum_{|\alpha|=k}u_{p,\alpha}\right)}.\nonumber
\end{eqnarray}

We will show that, if $k$ is large enough,  then we can find germs of analytic functions $f_{j,k}(z)$, $1\leq j\leq r$, vanishing at the origin, such that the formal solution to equation $(\ref{eq-main})$, $F=(F_1^{\leq m_1+k}+F_1^{>m_1+k},\ldots, F_r^{\leq m_r+k}+F_r^{>m_r+k})$ satisfies $\widehat F_j^{>m_j+k}(z)\prec z^{m_j+k}f_{j,k}(z)$ for $1\leq j\leq r$. 
\begin{lem}\label{lem-deriv}
Let $Q = (q_1,..., q_n)$ multiindex such that $\vert Q \vert = l$. Then, there exits a positive constant $c_l$ such that for any power series $f,g$ of $n$ variables
\begin{enumerate} 
\item 
$\widehat{\frac{\partial ^lf}{\partial x^Q }}\prec \frac{d^l\widehat f}{dz^i}$, 
$\widehat{\sum_{|Q|=l}\frac{\partial^l f}{\partial x^{Q}}}\prec c_l \frac{d^l\widehat f}{dz^l}$, 
\item $\widehat{fg}(z)\prec \widehat{f}(z)\widehat{g}(z)$.
\end{enumerate}
\end{lem}
\begin{proof}
Let us write $f=\sum f_{\al}x^{\al}$. We recall that $\|f^{(k)}\|=\sum_{|\alpha|=k}|f_{\alpha}|$. Hence, $\hat f(z)=\bar f(z,\ldots, z)$.
So, the last point follow from the fact that $\overline{fg}\prec \bar f\bar g$.
We have
$$
\frac{\partial^l f}{\partial x^{Q}}= \sum \frac{\al_1!\cdots\al_n!}{(\al_1-q_1)!\cdots (\al_n-q_n)!}f_{\al}x^{\al-Q}.
$$
We have, if $0\leq m\leq n$ and $0\leq k \leq n$, then 
$$
(1+t)^m(1+t)^{n-m}=\left(\sum_{i=0}^mC_m^it^i\right)\left(\sum_{j=0}^{n-m}C_{n-m}^jt^j\right)=\sum_{k=0}^n\left(\sum_{\scriptstyle{i+j=k\atop 0\leq i\leq m,\; 0\leq j\leq n-m}} C_m^iC_{n-m}^j\right)t^k,
$$
where $C_n^k=\frac{n!}{k!(n-k)!}$. Therefore, we have 
$$
C_n^k=\sum_{\scriptstyle{i+j=k\atop 0\leq i\leq m,\; 0\leq j\leq n-m}} C_m^iC_{n-m}^j. 
$$
As a consequence, we obtain by induction on $n\geq 1$, that
\begin{equation}\label{factorial}
\frac{\al_1!\cdots\al_n!}{(\al_1-q_1)!\cdots (\al_n-q_n)!}\leq \frac{|\al|!}{(|\al|-|Q|)!}\frac{q_1!\cdots q_n!}{|Q|!}\leq  \frac{|\al|!}{(|\al|-|Q|)!}.
\end{equation}
Therefore for any $p\ge 0$ one has
$$
\left\|\left(\frac{\partial^l f}{\partial x^{Q}}\right)^{(p)}\right\|\leq \left\|f^{(p+i)}\right\|\frac{(p+l)!}{p!}=\left(\frac{d^{l}\widehat{f}}{dz^{l}}\right)^{(p)}.
$$
which proves the first part of the first point. 
The second part is obtain as follow
\begin{eqnarray*}
\widehat{\sum_{|Q|=l}\frac{\partial^l f}{\partial x^{Q}}}(z)=\sum_{|Q|=l}\frac{\partial^l \bar f}{\partial x^{Q}}(z,\ldots,z) & = &\sum_{\alpha}\sum_{|Q|=l} \frac{\al_1!\cdots\al_n!}{(\al_1-q_1)!\cdots (\al_n-q_n)!}|f_{\al}|z^{|\al|-l}\\
&\prec &\sum_{p\geq l}\sum_{|\alpha|=p}\sum_{|Q|=l} \frac{p!}{(p-l)!}\frac{q_1!\cdots q_n!}{l!}|f_{\al}|z^{p-l}\\
&\prec &\sum_{p\geq l} \frac{p!}{(p-l)!}\left(\sum_{|Q|=l}\frac{q_1!\cdots q_n!}{l!}\right)\sum_{|\alpha|=p}|f_{\al}|z^{p-l}\\
&\prec &c_l\sum_{p\geq l} \frac{p!}{(p-l)!}\|f_{\al}^{(p)}\|z^{p-l}\\
&\prec &c_l\frac{d^{l}\widehat{f}}{dz^{l}}(z)
\end{eqnarray*}
Here we have set $c_l:=\left(\sum_{|Q|=l}\frac{q_1!\cdots q_n!}{l!}\right)$ if $l>0$ and $c_0:=1$.
\end{proof}
\begin{remark}\label{belit-lemma}
The previous lemma holds true if the norm of definition \ref{def-norms} is replaced by
the modified Belitskii norm~: indeed, we have
\begin{eqnarray*}
\left\|\frac{\partial^l f^{(p+l)}}{\partial x^{Q}}\right\|^2&= &\sum \left(\frac{\al_1!\cdots\al_n!}{(\al_1-q_1)!\cdots (\al_n-q_n)!}\right)^2|f_{\al}|^2\|x^{\al-Q}\|^2\\
&=&\sum_{\al} \left(\frac{\al_1!\cdots\al_n!}{(\al_1-q_1)!\cdots (\al_n-q_n)!}\right)^2|f_{\al}|^2\frac{(\al-Q)!}{(|\al|-|Q|)!}\\
&= & \sum \frac{\al_1!\cdots\al_n!}{|\al|!}|f_{\al}|^2\frac{\al_1!\cdots\al_n!|\al|!}{(\al-Q)!(|\al|-|Q|)!}\\
&\leq & \|f^{(p+l)}\|^2 \left(\frac{|\al|!}{(|\al|-|Q|)!}\right)^2.
\end{eqnarray*}
The last inequality is due to $(\ref{factorial})$. The product inequality is proved in \cite{stolo-lombardi}[proposition 3.6].

\end{remark}

Therefore, if $\widehat F_j^{>m_j+k}(z)\prec z^{m_j+k}f_{j,k}(z)$ and $0\leq l\leq m_j$, then 
$$
\widehat{\sum_{|\alpha|=l}\frac{\partial^{|\alpha|}F_j^{>m_j+k}}{\partial x^{\alpha}}}\prec c_l \frac{d^l  z^{m_j+k}f_{j,k}}{dz^l}(z).
$$
This leads us to define the following germ, at the origin, of analytic function of $(m_1+1)+\cdots +(m_r+1)+1$ variables~:
\begin{eqnarray}
G(z,z_{j,l},\,1\leq j\leq r,\;0\leq l\leq m_p)&:=& \frac{M}{1-cnz}\left[(nz)^{k+q+1}+\sum_{j=1}^r\sum_{l=0}^{m_p}(nz)^{p_{j,l}}c_l z_{j,l} \right]\nonumber\\
&& + \frac{M\left( \sum_{j,j'=1}^r\sum_{\scriptsize 1\leq l\leq m_j\atop 1\leq l'\leq m_{j'}}c_lc_{l'}z_{j,l}z_{p',l'}\right)}{1-c\left(nz+\sum_{j=1}^r\sum_{l=0}^{m_r}c_l z_{j,l}\right)}\label{def-G}
\end{eqnarray}

Let us consider the following system of analytic nonlinear differential operators
\begin{equation}\label{equ-f}
{\mathcal L}_i(f):=\frac{d^{m_i}z^{q+m_i+k}f_{i}}{dz^{m_i}} (z) - CG\left(z,\frac{d^lz^{m_j+k}f_{j}}{dz^l}(z),\;1\leq j\leq r,\; 0\leq l\leq m_j\right). 
\end{equation}
Here, the constant $C$ is the one defined by the big denominators property and $f=(f_{1},\ldots, f_r)$ is the unknown function, vanishing at the origin. 

Let $F$ be the formal solution of $F={\mathcal S}^{-1}\circ \pi ({\mathcal T}(F))$ as defined previously. Let us prove that, for all $1\leq j\leq r$, $\widehat F_j^{> m_j+k}\prec z^{m_j+k}f_j$. We recall that $f_j$ vanishes at the origin.
 
Let us prove by induction on the degree $i\geq m_j+k+1$, $1\leq j\leq r$, of the Taylor expansion that $\|F^{(i)}_j\|\leq f_{i-m_j-k}$. We recall that 
$F^{>k}=(F_1^{>m_1+k},\ldots, F_r^{>m_r+k})$. According to equation $(\ref{equation-k})$, we have
$$
\cS(F^{>k})= \pi\left(\cT(F)-\cS(F^{\leq k})\right).
$$

According to the big denominators property, we have
\begin{equation}\label{BD-action}
\frac{d^{m_j} \left(z^{q}\widehat F_j^{>m_j+k}\right)}{dz^{m_j}}\prec C\widehat{\mathcal T(F)-\cS(F^{\leq k})},
\end{equation}
where
$$
\widehat{\mathcal T(F)}=\sum_{i\geq 0}\|\left(\mathcal T(F)\right)^{(i)}\|z^i.
$$
We recall that $\left(\mathcal T(F)\right)^{(i)}$ denotes the homogeneous (vector) component of degree $i$ in the Taylor expansion at $0$ of $\mathcal T(F)$. Its the norm is the maximum of the norms of its coordinates component (see definition \ref{def-norms}).
According to $(\ref{majorant})$, we have
\begin{eqnarray*}
\overline{W_i\left(x,\partial F\right)-\cS_i(F^{\leq k})} & \prec & \frac{M}{1-c(x_1+\cdots x_n)}\left[ (x_1+\cdots x_n)^{k+q+1}+\right. \\
&&\left.\sum_{j=1}^r\sum_{|\alpha|\leq m_j}(x_1+\cdots x_n)^{p_{j,|\alpha|}}\frac{\partial^{|\alpha|} \bar F_j^{>k}}{\partial x^{\alpha}}\right]\nonumber\\
& &+\frac{M\left( \sum_{j,j'=1}^r\sum_{|\alpha|,|\alpha'|\leq m_j}\frac{\partial^{|\alpha|} \bar F_j^{>k}}{\partial x^{\alpha}}\frac{\partial^{|\alpha'|} \bar F_{j'}^{>k}}{\partial x^{\alpha'}}\right)}{1-c\left(x_1+\ldots +x_n+\sum_{p=1}^r\sum_{k=0}^{m_r}\sum_{|\alpha|=k}u_{p,\alpha}\right)}.
\end{eqnarray*}


Let us set, for $\leq j\leq r$, $k+1\leq K$~:
$$
\widehat F_j^{m_j+k+1,m_j+K}(z) := \sum_{i=m_j+k+1}^{m_j+K}\|F_j^{(i)}\|z^i,\quad f^{\leq K}(z) := \sum_{i=1}^{K}f_iz^i
$$
Furthermore, we shall set, for convenience 
$$
\widehat F^{k+1, K}:=(\widehat F_1^{m_1+k+1,m_1+K},\ldots, \widehat F_r^{m_r+k+1,m_r+K} ).
$$
Let us assume that, by induction on $m\geq 1$, that we have 
$$
\widehat F_j^{m_j+k+1,m_j+k+m}\prec z^{m_j+k}f^{\leq m}.
$$
Let us prove that $\widehat F_j^{m_j+k+1,m_j+k+m+1}\prec z^{m_j+k}f^{\leq m+1}$.
Hence, we have for all $1\leq j\leq r$ and nonnegative integer $p\leq m_j$, 
$$
\frac{d^p \widehat F_j^{m_j+k+1,m_j+k+m}}{dz^p}(z)\prec \frac{d^p z^{m_j+k}f^{\leq m}}{dz^p}(z).
$$ 
According to lemma \ref{lem-deriv}, we have
$$
\sum_{|\alpha|=p}\frac{\partial^{|\alpha|} \bar F_{j }^{m_j+k+1,m_j+k+m}}{\partial x^{\alpha}}(z,\ldots,z)\prec c_p\frac{d^p \widehat F_j^{m_j+k+1,m_j+k+m}}{dz^p}(z)\prec c_p \frac{d^p z^{m_j+k}f^{\leq m}}{dz^p}(z).
$$
According to the definition $(\ref{def-G})$ of $G$ and estimate $(\ref{majorant})$, we obtain, for $1\leq i\leq r$~:
\begin{equation}
\overline{W_i\left(x,\partial F^{k+1,k+m}\right)-\cS_i(F^{\leq k})}(z,\ldots,z)\prec G\left(z,\frac{d^lz^{m_j+k}f_{j}^{\leq m}}{dz^l}(z),\;1\leq j\leq r,\; 0\leq l\leq m_j\right)
\end{equation}
On the other hand, according to $(\ref{BD-action})$ and taking the (k+m+q+1)-jet at the origin, we obtain
$$
\frac{d^{m_j} \left(z^{q}\widehat F_j^{m_j+k+1, m_j+k+m+1}\right)}{dz^{m_j}}=\left(\frac{d^{m_j} \left(z^{q}\widehat F_j^{>m_j+k}\right)}{dz^{m_j}}\right)^{\leq k+m+q+1}\prec C\left(\widehat{\mathcal T(F)-\cS(F^{\leq k})}\right)^{\leq k+m+q+1}.
$$
Since both $\mathcal T$ and $G$ strictly increase the order by $q$, then we have
$$
\left(\mathcal T(F)\right)^{\leq k+m+q+1}=\left(\mathcal T(F^{\leq k+m})\right)^{\leq k+m+q+1}.
$$
Indeed, the order at the origin of $\cT(F)-cT(F^{\leq k+m})$ is strictly greater than $k+m+1+q$.
Furthermore, the $(k+q+m+1)$-jet of $G\left(z,\frac{d^lz^{m_j+k}f_{j}}{dz^l}(z),\;1\leq j\leq r,\; 0\leq l\leq m_j\right)$ at the origin is nothing but $(k+q+m+1)$-jet of $G\left(z,\frac{d^lz^{m_j+k}f_{j}^{\leq m}}{dz^l}(z),\;1\leq j\leq r,\; 0\leq l\leq m_j\right)$ at the origin.

Combining all these last estimates, we obtain
\begin{eqnarray*}
\frac{d^{m_j} \left(z^{q}\widehat F_j^{m_j+k+1, m_j+k+m+1}\right)}{dz^{m_j}}&\prec &C\left(\widehat{\mathcal T(F)-\cS(F^{\leq k})}\right)^{\leq k+m+q+1}\\
&= & C\left(\widehat{T(F^{\leq k+m})-\cS(F^{\leq k})}\right)^{\leq k+m+q+1}\\
& \prec & \left(G\left(z,\frac{d^lz^{m_j+k}f_{j}^{\leq m}}{dz^l}(z),\;1\leq j\leq r,\; 0\leq l\leq m_j\right)\right)^{\leq k+m+q+1}\\
&= &  \left(G\left(z,\frac{d^lz^{m_j+k}f_{j}}{dz^l}(z),\;1\leq j\leq r,\; 0\leq l\leq m_j\right)\right)^{\leq k+m+q+1}\\
&= & \frac{d^{m_i} \left(z^{q+m_i}f^{\leq k+m+1}\right)}{dz^{m_i}}
\end{eqnarray*}
Hence, we have shown that
$$
\frac{d^{m_j} \left(z^{q}\widehat F_j^{m_j+k+1, m_j+k+m+1}\right)}{dz^{m_j}}\prec \frac{d^{m_i} \left(z^{q+m_i}f^{\leq k+m+1}\right)}{dz^{m_i}}.
$$ 
Comparing the coefficient of $z^{k+m+q+1}$, we obtain
$$
\|F_j^{(m_j+k+m+1)}\|\leq f_{k+m+1},
$$
and we are done.

\subsubsection{Existence of an analytic solution of the the majorizing system}\label{sec-sol-maj}

In this section, we shall prove that system of differential equation $(\ref{diff-equ})$ has an analytic solution vanishing a the origin. We shall show that this system has a kind of {\bf regular singularity} at the origin. In fact, if there were only one big denominator, that is $r=1$, then the system reduces indeed to a differential equation with regular singularity at the origin. To prove the existence of an analytic solution, we shall apply a modification of the argument used by B. Malgrange in order to prove the existence of analytic solution of a regular differential equation \cite{malgrange-maillet}. The way to proceed is to rescale the equation by a scalar $\la$ and then to prove that solving the system of differential equation corresponds to find a solution of an implicit function theorem in a suitable Banach space of analytic functions.

Let us rescale these equations by the mean of the map $t\mapsto \la t$. Let us denote $g_{\lambda}(t):=g(\lambda t)$. Then we have $\la(\frac{d g}{dz})_{\la}=\frac{d g_{\la}}{dz}$. Hence, we have 
\begin{equation}\label{equ-lambda}
({\mathcal L}_i(f))_{\la}=\la^{q+k}\frac{d^{m_i}z^{m_i+q+k}(f_i)_{\la}}{dz^{m_i}} (z) - CG\left(\la z,\la^{m_j+k-l}\frac{d^l z^{m_j+k}(f_{j})_{\la}}{dz^l},\;1\leq j\leq r,\; 0\leq l\leq m_j\right)
\end{equation}
We have, with a short notation,
\begin{eqnarray*}
G\left(\la z,\la^{m_j+k-l}\frac{d^l z^{m_j+k}(f_{j})_{\la}}{dz^l}\right) &= & \frac{M}{1-cn\la z}\left[\la^{k+q+1}(nz)^{k+q+1}\right.\\
&&+\left.\sum_{j=1}^r\sum_{l=0}^{m_p}\la^{m_j+p_{j,l}+k-l}(nz)^{p_{j,l}}c_l \frac{d^l z^{m_j+k}(f_{j})_{\la}}{dz^l} \right]\\
&& + A_2(\la)
\end{eqnarray*}
where we have written~:
$$
A_2(\la):=\frac{M\left( \sum_{j,j'=1}^r\sum_{\scriptsize 1\leq l\leq m_j\atop 1\leq l'\leq m_{j'}}c_lc_{l'}\la^{m_{j'}+m_j+2k-l-l'}\frac{d^l z^{m_j+k}(f_{j})_{\la}}{dz^l}\frac{d^{l'} z^{m_{j'}+k}(f_{j'})_{\la}}{dz^{l'}}\right)}{1-c\left(n\la z+\sum_{j=1}^r\sum_{l=0}^{m_r}c_l \la^{m_j+k-l}\frac{d^l z^{m_j+k}(f_{j})_{\la}}{dz^l}\right)}.
$$
According definition $(\ref{def-p})$ of $p_{j,l}$, we have $m_j+p_{j,l}+k-l-(q+k)\geq 1$. Furthermore, assume $k$ is large enough, that is such that
\begin{eqnarray}
m_{j'}+m_j+2k-l-l'-(q+k)&\geq & 1\label{ordre-prod}\\
m_j+k-l&\geq& 0.\nonumber
\end{eqnarray}
It is sufficient that
\begin{equation}
k\geq q+1.\label{cond-k}
\end{equation}
 As a consequence, dividing ${\cL_i(f)}_{\la}$ by $\la^{q+k}$, we obtain the following system of differential operators~:
$$
L_i(f_{\la}):=\frac{d^{m_i}z^{m_i+q+k}(f_i)_{\la}}{dz^{m_i}} -\la \tilde G\left(\la, z,\frac{d^l z^{m_j+k}(f_{j})_{\la}}{dz^l},\;1\leq j\leq r,\; 0\leq l\leq m_j\right)
$$
for some analytic $\tilde G$, in all its variables, at the origin.

Let us consider, for $m\geq 0$, the space $H^m$ of power series $f=\sum_{k\geq 1}a_lz^l$ such that
$$
\|f\|_m:=\sum_{l\geq 1}|a_l|l^m<+\infty.
$$
There exists a constant $C$, such that for any $f\in H^{m_i}$, then
$$
\left\|\frac{d^{m_i}z^{m_i+q+k}f}{dz^{m_i}}\right\|_0=\sum_{l\geq 1}|f_l|(m_i+l+q+k)\cdots (1+l+q+k)\leq C\|f\|_{m_i}.
$$ 
This means that $\frac{d^{m_i}z^{m_i+q+k}f}{dz^{m_i}}\in H^0$. Furthermore, since $H^m\hookrightarrow H^{m'}$ for $m'\leq m$, then $\frac{d^{l}z^{m_i+k}f}{dz^{l}}\in H^0$. On the other hand, according to lemma \ref{lem-gevrey}, $H^0$ is a Banach algebra since $\|fg\|_0\leq \|f\|_0\|g\|_0$. As a consequence, we can consider the analytic mapping 
\begin{eqnarray*}
L(\la, f_1,\ldots,f_r)& :& (\Bbb R,0)\times H^{m_1}\times\cdots\times H^{m_r}\rightarrow  (H^0)^r\\
(\la, f_1,\ldots,f_r)&\mapsto& \left(\frac{d^{m_i}z^{m_i+q+k}f_i}{dz^{m_i}} -\la \tilde G\left(\la, z,\frac{d^l z^{m_j+k}f_{j}}{dz^l},\;1\leq j\leq r,\; 0\leq l\leq m_j\right)\right)_{i=1,\ldots, r}.
\end{eqnarray*}
We have $L(0)=0$. Moreover, the differential of $L$ with respect to $(f_1,\ldots,f_l)$, at the origin, in the direction $v=(v_1,\ldots, v_r)$ is 
$$
D_fL(0)v= \left(\frac{d^{m_i}z^{m_i+q+k}v_i}{dz^{m_i}} \right)_{i=1,\ldots, r}.
$$
It is invertible from $H^{m_1}\times\cdots\times H^{m_r}$ on the subspace of $(H^0)^r$ of elements of order $\geq k+q+1$. According to the analytic implicit function theorem, there exists a map 
$$
\la\in (\Bbb R,0)\mapsto (f_{1}(\la),\ldots,f_r(\la))\in H^{m_1}\times\cdots\times H^{m_r}
$$ 
such that $(f_{1}(0),\ldots,f_r(0))=0$ and $L(\la, f_1(\la),\ldots,f_r(\la))=0$ for $\la$ small enough.

\subsection{Proof of the variation of the main theorem}
Let us sketch briefly how to prove theorem \ref{var-main-thm}. The construction of the formal solution $F=(F_1,\ldots, F_r)$ is the same as above up to shift of the order by $d$. We look for functions $f_j$ vanishing at the origin such that $\widehat F_j\prec z^{m_j+d}f_j$. They will solves the majorizing system
\begin{equation}\label{d-equ-f}
{\mathcal L}_i(f):=\frac{d^{m_i}z^{q+m_i+k}f_{i}}{dz^{m_i}} (z) - CG\left(z,\frac{d^lz^{m_j+d+k}f_{j}}{dz^l}(z),\;1\leq j\leq r,\; 0\leq l\leq m_j\right). 
\end{equation}
Now, equation $\ref{equ-lambda}$ becomes
\begin{equation}\label{d-equ-lambda}
({\mathcal L}_i(f))_{\la}=\la^{q+k}\frac{d^{m_i}z^{m_i+q+k}(f_i)_{\la}}{dz^{m_i}} (z) - CG\left(\la z,\la^{m_j+k+d-l}\frac{d^l z^{m_j+k}(f_{j})_{\la}}{dz^l},\;1\leq j\leq r,\; 0\leq l\leq m_j\right)
\end{equation}
Since $\cT$ is $d$-regular, then we have $m_j+p_{j,l}+k+d-l-(q+k)\geq 1$. Therefore, $G\left(\la z,\la^{m_j+k+d-l}\frac{d^l z^{m_j+k}(f_{j})_{\la}}{dz^l},\;1\leq j\leq r,\; 0\leq l\leq m_j\right)$ is divisible by $\la^{q+k+1}$. We conclude as above.

\section{Relative big denominators and Gevrey classes}\label{sec-Gevrey}

In this section, we investigate what happens when the denominators do not growth fast enough to overcome the divergence generated by the differentials of the unknowns involved. We use the same notation as in section \ref{sec-PDEs}~: we are given ${\bf m}=(m,\ldots,m)\in \Bbb N^r$ and we consider $\mathcal S,\mathcal T,\pi$ be maps as in theorem \ref{main-thm}. The equation to be solved is still $(\ref{eq-main})$. We will assume that all item of the big denominators property holds but $(\ref{BD})$ which is replaced by the weaker condition~: for each $1\leq j\leq r$, there exists $0<\alpha$ such that 
\begin{equation}\label{notsobig}
z^q\widehat{\mathcal S^{-1}\pi G}(z) \prec \ z^{m}C\sum_{i\geq q+1} \frac{\|G^{(i)}\|}{i^{m-\alpha}}z^i.
\end{equation}
for some constant $C$ independent of $G$. In that case, we will say that the triple $(\mathcal S, \mathcal T,\pi )$  has the {\it relative big denominator property of order $m- \alpha<m$}. In that case, estimate $(\ref{BD-expand})$ is replaced by
\begin{equation}\label{notsoBD-expand}
\left\|\left(\cS_j^{-1}\circ \pi (G)\right)^{(i+m)}\right\|\leq C\frac{\|G^{(i+q)}\|}{i^{m-\alpha}}.
\end{equation}

In that situation, we cannot expect convergence of the solution. We will show that the problem has a power series solution that diverges in a controlled way.
\begin{definition}Let $\alpha\geq 0$. We say that a formal power series $f=\sum_{Q\in\Bbb N^n}f_Qx^Q$ is $\alpha$-{\it Gevrey} if there exists constants $M,C$ such that for all $Q\in\Bbb N^n$, $|f_Q|\leq MC^{Q}(|Q|!)^{\alpha}$.
\end{definition}
\begin{thm}
\label{thm-Gevrey}
Let $(\mathcal S,\mathcal T,\pi)$ be maps as in theorem \ref{main-thm}. Assume that the triple $(\mathcal S, \mathcal T,\pi )$ has the {\it relative big denominator property of order $m-\alpha< m$}. 
Then, equation (\ref{eq-main}) has an $\alpha$-Gevrey formal solution $F$.
\end{thm}
\begin{remark}
For the problem of conjugacy of germs of vector fields vanishing at a point (say $0$), A.D. Brjuno proved that, provided that the (assumed to be) diagonal linear vector field $S$ satisfies a small divisors condition, and if an nonlinear analytic perturbation $R$ satisfies Brjuno's condition A, then $X:=S+R$ is analytically conjugate to a normal form. If condition A is not satisfied then analytic counter examples were constructed and one could see on these counter examples, that divergence of the transformation is of Gevrey type. Theorem \ref{thm-Gevrey} shows that, in a much more general context, the expected divergence is always Gevrey.

This Gevrey character of the formal transformation leads, in the case of vector fields, to the approximation of the dynamics by an analytic normal form up to an exponentially small remainder \cite{stolo-lombardi}[theorem 6.11] (see also \cite{ramis-schafke} for averaging).
\end{remark}
\begin{proof}
Our proof is very inspired by Malgrange's version of Maillet theorem \cite{malgrange-maillet} which proves the Gevrey character of formal solutions of holomorphic nonlinear differential equations of one variable with irregular singularity. Let us define the Banach space
$$
H^{s,\beta}:=\left\{f=\sum_{i\geq 0}f_it^i\in\Bbb C[[t]]\;|\; \|f\|_{s,\beta}:=\sum_{i\geq 0}\frac{|f_i|i^\beta}{(i!)^{s}}<+\infty\right\}.
$$
Let us set
$$
\delta := t\frac{\partial}{\partial t}.
$$
Let us start with an elementary lemma~:
\begin{lem}\label{lem-gevrey}
\begin{enumerate}
\item $\partial_t^k(t^n\psi)=t^{n-k}(\delta+n)\cdots(\delta +n-k+1)\psi$
\item $t\psi\in H^{s,\beta} \Leftrightarrow \psi\in H^{s,\beta-s}$ if $\psi(0)=0$.
\item if $\be\geq s$, then $\|\psi\|_{s, \be-s}\leq \|t\psi\|_{s, \be}$.
\item Let $g\in H^{s,\beta}$ then $t^q\delta^pg\in H^{s,0}$ whenever $q\geq \max(0, \frac{1}{s}(p-\beta))$.
\item Let $g\in H^{s,0}$  and $m\leq s$. Then $t^q\delta^pg\in H^{s,s-m}$ whenever $p\leq (q-1)s+m$.
\item $\|fg\|_{s,0}\leq \|f\|_{s,0}\|g\|_{s,0}$.
\item Let $a>0$ and $f,g\in H^{s,a}$. If $f(0)=g(0)=0$ then $fg\in H^{s,a}$.
\end{enumerate}
\end{lem}
\begin{proof}
\begin{enumerate}
\item
The proof is obtained by induction on $k$ since
$$
\partial_t(t^n\psi)=nt^{n-1}\psi+t^n\partial_t(\psi)= t^{n-1}(\de+n)\psi.
$$
\item Let us first assume that $\be\geq s$. Let $t\psi\in H^{s,\beta}$, then
\begin{equation}\label{inequality-gevrey}
\|\psi\|_{s,\beta-s}\leq\sum_{i\geq 1}\frac{|\psi_i|(i+1)^{\be-s}}{i!^s}=\sum_{i\geq 1}\frac{|\psi_i|(i+1)^{\be}}{(i+1)!^s}=\|t\psi\|_{s,\beta}<+\infty.
\end{equation}
Moreover, for $i\geq 1$,
$$
(1/2)^{\be-s}(i+1)^{\be-s}\leq i^{\be-s}\leq (i+1)^{\be-s}.
$$
Hence,
$$
(1/2)^{\be-s}\|x\psi\|_{s,\be}\leq  \|\psi\|_{s,\be-s}\leq \|x\psi\|_{s,\be}.
$$

In the case, $\be-s<0$, we have
$$
\|t\psi\|_{s,\beta}=\sum_{i\geq 1}\frac{|\psi_i|(i+1)^{\be}}{(i+1)!^s}\leq \sum_{i\geq 1}\frac{|\psi_i|i^{\be-s}}{i!^s}= \|\psi\|_{s,\beta-s}.
$$
On the other hand, if $i\geq 1$ then $1/i\leq 2/(i+1)$. Hence,
$$
\|\psi\|_{s,\be -s}=\sum_{i\geq 1}\frac{|\psi_i|i^{\be-s}}{i!^s}\leq 2^{s-\be}\sum_{i\geq 1}\frac{|\psi_i|(i+1)^{\be}}{(i+1)!^s}=2^{s-\be}\|t\psi\|_{s,\beta}.
$$
\item Inequality $(\ref{inequality-gevrey})$ holds if $\psi(0)\neq 0$ and $\be-s\geq 0$.
\item We have $t^q\de^p(\sum_{i\geq 1}g_it^i)= \sum_{i\geq 1}g_ii^pt^{i+q}$.So,
$$
\|t^q\de^pg\|_{s,0}= \sum_{i\geq 1}\frac{|g_i|i^p}{(i+q)!^{s}}.
$$
Since $\frac{1}{[(i+1)\cdots (i+q)]^s}\leq \frac{1}{i^{qs}}$ then we have $\frac{i^p}{(i+q)!^s}\leq \frac{i^{\be}}{(i)!^s}$ if $q\geq \max(0, \frac{1}{s}(p-\beta))$. Thus, under that condition,
we have $\|t^q\de^pg\|_{s,0}\leq \|g\|_{s,\be}$.
\item Indeed, we have
$$
\|t^q\de^pg\|_{s,s-m}= \sum_{i\geq 1}\frac{|g_i|i^p(i+q)^{s-m}}{(i+q)!^{s}}.
$$
Since
$$
\frac{i^p(i+q)^{s-m}}{(i+1)\cdots(i+q)^{s}}\leq \frac{i^p}{(i+q)^{m}(i+1)\cdots(i+q-1)^{s}}\leq \frac{i^p}{i^{m+(q-1)s}}\leq 1
$$
then we have $\|t^q\de^pg\|_{s,s-m}\leq \|g\|_{s,0}$.
\item We have
\begin{eqnarray*}
\|fg\|_{s,0}&\leq &\sum_{i\geq 0}\sum_{k=0}^{i}\frac{|f_k|}{k!^s}\frac{|g_{k-i}|}{(k-i)!^s}\frac{k!^s(k-i)!^s}{i!^s}\\
&\leq & \sum_{i\geq 0}\sum_{k=0}^{i}\frac{|f_k|}{k!^s}\frac{|g_{k-i}|}{(k-i)!^s}= \|f\|_{s,0}\|f\|_{s,0}.
\end{eqnarray*}
\item Since $i\leq 2k(i-k)$ if $1\leq k\leq i-1$ and $i\geq 2$ then
$$
\|fg\|_{s,a}\leq \sum_{i\geq 2}\sum_{k=1}^{i-1}\frac{|f_k|k^a}{k!^s}\frac{|g_{k-i}|(i-k)^a}{(k-i)!^s}\frac{k!^s(k-i)!^s}{i!^s}.\left(\frac{i}{k(i-k)}\right)^a\leq 2^a\|f\|_{s,a}\|g\|_{s,a}.
$$

\end{enumerate}
\end{proof}

The equation to be solved is equation $(\ref{eq-main})$. As above, that is as in the ``big denominators case'', this equation has a formal solution $F$. As above, we select an integer $k\geq 1$ and we look for a majorant of the solution $F^{>k}=( F_1^{>m_1+k},\ldots, F_r^{>m_r+k})$ of
\begin{equation}\label{equ-k}
\cS(F^{>k})=\pi\left(\cT(F)-\cS(F^{\leq k})\right).
\end{equation}
Let us define the following operator $L$ that maps a formal power series of one variable to~: if $f=\sum_{i\geq 1}f_it^i\in\Bbb C[[t]]$, then
\begin{equation}\label{op-L}
L(f)(z):=\sum_{i\geq 1}f_iz^{i}i^{m-\alpha}.
\end{equation}
According to $(\ref{notsoBD-expand})$, we have
\begin{eqnarray*}
L(z^q\widehat{\cS^{-1}_j\pi(G)})&=&\sum_{i\geq 1}\|(\cS_j^{-1}\pi(G)^{(i+m)}\|z^{i+q+m}(q+i+m)^{m-\alpha}\\
&\prec& C\sum_{i\geq 1}\|G^{(i+q)}\|z^{i+q+m}=z^{m}\widehat G
\end{eqnarray*}
Hence, equation $(\ref{equ-k})$ leads to
\begin{equation}\label{notsoBD-action}
L(z^q\widehat{F_j^{>m_j+k}})\prec Cz^{m}\left(\widehat{\cT(F)-\cS(F^{\leq k})}\right).
\end{equation}
Let $G$ be the function defined by $(\ref{def-G})$. 
As above, we show by induction on the truncation order, that there exists a formal power series $f_j$, vanishing at the origin, such that 
$$
\widehat F_j^{>m+k}\prec z^{m+k}f_j 
$$
and such that the $f=(f_1,\ldots, f_r)$s solves the system
\begin{equation}\label{equ-f-gevrey}
{\mathcal L}_i(f):=L(z^{q+m+k}f_{i}) (z) - Cz^{m}CG\left(z,\frac{d^lz^{m+k}f_{j}}{dz^l}(z),\;1\leq j\leq r,\; 0\leq l\leq m\right)=0. 
\end{equation}
The only change in the proof is that we are using $(\ref{notsoBD-action})$ instead of $(\ref{BD-action})$.

Let us show that such a solution $f_j$ is a ${\alpha}$-Gevrey.

Let us first consider the case $\al\leq m$. Let us show that equation $(\ref{equ-f-gevrey})$ has a unique solution $f_j\in H^{\al,0}$. In order to prove this, we shall show that ${\bf f}=(f_1,\ldots,f_j)\in H^{\al,0}\times\cdots\times H^{\al,0}$ is solution of an analytic implicit equation in the Banach space $H^{\al,0}$.

Let $s,\be\geq (m+k+q)s$ be non-negative numbers. Let us set $g:=t^{m+k+q}f$, where $f$ stands for one the $f_i$'s. According to the third point of the previous lemma, if $g=t^{m+k+q}f\in H^{s,\be}$ then $f\in H^{s, \be-(m+k+q)s}$. Let us write $L(g)=t^{m+k+q}\psi$. We have
$$
\|L(g)\|_{s,\beta-m+\alpha}=\sum_{i\geq m+q+m+1}|g_i|\frac{i^{\beta}}{(i)!^s}=\|g\|_{s,\beta}.
$$
Hence, if $g \in H^{s, \be}$ then $L(g)=t^{q+m+k}\psi\in H^{s,\be-m+\al}$. Again, if $\be-m+\al\geq (m+k+q)s$, then $\psi\in H^{s,\be-(m+q+k)s-m+\al}$. Let us set
\begin{equation}\label{beta}
\be:=(m+k+q)s+m-\al.
\end{equation}
Therefore, if $m\geq \al$, then $\beta\geq (m+k+m)s$. In that case, if $g=t^{m+q+k}f\in H^{s,\be}$ we then have $f\in H^{s, m-\al}$ and $\psi\in H^{s,0}$. According the first point of lemma \ref{lem-gevrey}, there are universal coefficients $c_{p,j}$, $0\leq p\leq m$, $0\leq j\leq p$, such that
\begin{equation}\label{derivees}
\frac{d^p t^{m+k}f}{d t^p}=t^{m+k-p}\sum_{j=0}^pc_{p,j}\de^jf.
\end{equation}
Moreover, we have $t^a\de^jf\in H^{s,0}$ as soon as $a\geq \max(0, \frac{1}{s}(j-(m-\al))$. Let us show that if we set $s:=\al$, then $t^{m+k-p}\de^jf\in H^{s,0}$ for $0\leq j\leq p\leq m$.
Indeed, let $[\al]$ denotes the integer part of $\al$, i.e. $[\al]\leq \alpha<[\alpha]+1$. If $0\leq j\leq m-([\al]+1)$, then $j-m+\al\leq -([\al]+1)+\al<0$. Hence, $\max(0, \frac{1}{s}(j-(m-\al))=0$ so
$t^{m+k-p}\de^jf\in H^{s,0}$ for all $j\leq k\leq m$. If $m-[\al]\leq j\leq k\leq m$ then $(j-m+\al)/s\leq \al/s$. On the other hand, we have $m+k-p\geq k\geq 1$. Therefore, if $s:=\alpha$ then
$m+k-p\geq \max(0, \frac{1}{s}(j-(m-\al))$ for all $m-[\al]\leq j\leq k\leq m$.

As a consequence, we have~:
\begin{itemize}
\item if $p\leq m-[\al]-1$, then
\begin{equation}\label{cas1}
\frac{d^p (t^{m+k}f)}{dt^p}=t^{m+k-p}\underbrace{\sum_{j\leq p}c_{p,j}\de^jf}_{\in H^{\al,0}}=:t^{m+k-p}g_p;
\end{equation}
\item if $m-[\al]\leq p\leq m$, then
\begin{equation}\label{cas2}
\frac{d^p (t^{m+k}f)}{d t^p}=t^{m+k-1-p}\underbrace{\sum_{j\leq p}c_{p,j}t\de^jf}_{\in H^{\al,0}}=:t^{m+k-1-p}(tg_p).
\end{equation}
\end{itemize}

Let us consider the system of equations $(\ref{equ-f-gevrey})$~:
$$
{\mathcal L}_i(f):=L(t^{q+m+k}f_{i}) (t) - Ct^{m}CG\left(t,\frac{d^lt^{m+k}f_{j}}{dt^l}(t),\;1\leq j\leq r,\; 0\leq l\leq m\right)=0.
$$
According to definition of $p_{j,l}$ (the mapping $\cT$ is regular by assumption), we have $p_{j,l}+m+k-1-l\geq q+l$. On the other hand, if $k$ is large enough, then $m_{j'}+m_j+2k-l-l'-(q+k)\geq  2$. In that case, $G\left(t,\frac{d^lt^{m+k}f_{j}}{dt^l}(z),\;1\leq j\leq r,\; 0\leq l\leq m\right)$ is not only divisible by $t^{q+k}$ but also it can be written as
$$
G\left(t,\frac{d^lt^{m+k}f_{j}}{dt^l}(t),\;1\leq j\leq r,\; 0\leq l\leq m\right)=t^{q+k}\tilde G\left(t,tg_{j,p},\;1\leq j\leq r,\; 0\leq p\leq m\right)
$$
where the $g_{j,p}$'s are defined by $(\ref{cas1})$ and $(\ref{cas2})$ for $f=f_j$ and where $\tilde G$ is an analytic function of all its arguments.
Let us rescale this equation by the mean of the map $t\mapsto \la t$. Let us denote $g_{\lambda}(t):=g(\lambda t)$. Then we have $(\de g)_{\la}=\de (g_{\la})$. According to the definition of the $p_{j,l}$'s and the constraint on $k$ then, as in section \ref{sec-sol-maj}, we obtain after dividing by $\la^{k+m+q}$~:
\begin{equation}\label{gevrey-rescaled}
L_i(f_{\la}):=L\left(t^{m+q+k}(f_i)_{\la}\right) -\la t^{m+q+k}\tilde G\left(\la, t,t(g_{j,p})_{\la},\;1\leq j\leq r,\; 0\leq p\leq m\right)
\end{equation}
for some analytic $\tilde G$, in all its variables, at the origin.

According to $(\ref{cas1}), (\ref{cas2})$ then if $f_j\in H^{\al, m-\al}$ then $tg_{j,p}\in H^{\al, 0}$ for all $p\leq m$. Then according to the last property of lemma \ref{lem-gevrey}, 
$$
\tilde G\left(\la, t,t(g_{j,p})_{\la},\;1\leq j\leq r,\; 0\leq p\leq m\right)\in H^{\al, 0}.
$$
Let us consider the analytic mapping
$$
A(\la, f):=\frac{1}{t^{m+k+q}}L(t^{m+k+q}f_i) - \la\tilde G\left(\la, t,tg_{j,p},\;1\leq j\leq r,\; 0\leq p\leq m\right)_{i=1,\ldots, r}
$$
from $(\Bbb C,0)\times \left(H^{\al,m-\al}\right)^r$ into $(H^{\al,0})^r$. 
We have $A(0,0)=0$ and the differential of $A$ with respect tp $f$ at the point $0$,$D_f A(0,0)$, is the linear mapping $f\mapsto \frac{1}{t^{m+k+q}}(L(t^{m+k+q}f_i))_{i=1,\ldots, r}$. This map is invertible from the subspace of $(H^{\al,m-\al})^{r}$ vanishing at the origin into the subspace of $(H^{\al,0})^r$ vanishing at the origin. By the implicit function theorem, for $\la$ small enough, there exists a curve $\la\mapsto f_{\la}\in (H^{\al,0})^r$ such that $f_{0}=0$ and such that for $(\ref{gevrey-rescaled})$ holds. We are done in the case $\al\leq m$.

Let us now assume that $\al>m$. Let us set $\beta:= (m+k+q)s$ instead of $(\ref{beta})$ and let us set $s:=\al$. Then if $g=t^{m+k+q}f\in H^{\al,\beta}$ then $f\in H^{\al,0}$ and $\psi\in H^{\al,\al-m}$ where $L(g)=t^{m+q+k}\psi\in H^{\al,(m+k+q)\al-m+\al}$.
According to the fifth point of lemma \ref{lem-gevrey}, $t\delta^jf\in H^{\al,\al-m}$ whenever $j\leq m$, which is the case. According to the last point of lemma \ref{lem-gevrey}, the multiplication of $tg_p$ by any (nonnegative) power of $t$ also belong to $H^{\al,\al-m}$. Hence, if $f\in H^{\al, 0}$ then we have 
$$
\tilde G\left(\la, t,t(g_{j,p})_{\la},\;1\leq j\leq r,\; 0\leq p\leq m\right)\in H^{\al, \al-m}.
$$
As above, we consider the analytic mapping
$$
A(\la, f):=\frac{1}{t^{m+k+q}}L(t^{m+k+q}f_i) - \la\tilde G\left(\la, t,tg_{j,p},\;1\leq j\leq r,\; 0\leq p\leq m\right)_{i=1,\ldots, r}
$$
from $(\Bbb C,0)\times \left(H^{\al,0}\right)^r$ into $(H^{\al,\al-m})^r$. Its differential $D_f A(0,0)$, is the linear mapping $f\mapsto \frac{1}{t^{m+k+q}}(L(t^{m+k+q}f_i))_{i=1,\ldots, r}$. It is invertible from the subspace of $(H^{\al,0})^{r}$ vanishing at the origin into the subspace of $(H^{\al,\al-m})^r$ vanishing at the origin. We conclude as above by the analytic implicit function theorem.

\end{proof}
\section{Applications}
\label{sec-appl}
\subsection{Singular vector fields with a fixed linear approximation}
\label{vf}

We consider the classical problem of classification of germs of vector field at a singular point $0\in \mathbb R^n$, with a fixed linear part $\dot x = Ax$ at the origin, with respect to germs of diffeomorphisms fixing the origin.
Following the notations of section \ref{sec-big-denom-classif-problems}, we set
$$
r = n, \ m_i = 1,\ 1\leq i\leq n, \ s = n, \ q = 1
$$ 
so that $\mathcal F_{r,\bf m} = \left(\mathbb A_n^n\right)_{>1}$ and $P^{(1)}(x) = Ax$. 
The action of the group of local diffeomorphisms is as follows:
$$(id + F(x))_*\left(Ax + R(x)\right) = \left(I + DF(x)\right)^{-1} \left(Ax + AF(x) + R(x+ F(x))\right).$$
It is an analytic differential action of order $1$. The linear operator $\mathcal S$ is
$$\mathcal S: \ \left(\mathbb A^n_n\right)_{>1} \to \left(\mathbb A^n_n\right)_{>1}, \ \ \mathcal S(F)(x) = AF(x) - DF(x)Ax$$
which is, in fact, the Lie bracket of the vector fields $\dot x = Ax$ and $\dot x = F(x)$. It increases the order by $1$ (we have to keep in mind that the space ${\cF}_{r,\bf m}^{(i)}$ is the space of homogeneous vector fields of degree $\bf i-1$).
Let us define
\begin{eqnarray*}
\mathcal T(F)(x)&=&\left((I + DF(x))^{-1}-(I - DF(x))\right)\left(Ax + AF(x) + R(x+ F(x))\right)\\
& & +R(x+F(x))-DF(x)(AF(x)+R(x+F(x)).
\end{eqnarray*}
Let us show that it is regular. Indeed, let $F,G$ be formal vector field of order $\geq 2$. We have (with a clear abuse of notation)
\begin{eqnarray*}
\frac{\partial \cT}{\partial F}(F)G (x)&=& \left((I + DF(x))^{-1}-(I - DF(x))\right)\left(AG(x) + DR(x+ F(x))G\right)\\
& & +DR(x+F(x))G-DF(x)(AG(x)+DR(x+F(x))G)
\end{eqnarray*}
Since $R$ is of order $\geq 2$, then $DR(x+ F(x))$, $DF(x)$ and $((I + DF(x))^{-1}-(I - DF(x))$ have order $\geq 1=p_{j,0}$ for any $1\leq j\leq n$. On the other hand, we have
\begin{eqnarray*}
\frac{\partial \cT}{\partial F'}(F)DG (x)&=& \left(\sum_{k\geq 2} (-1)^kk(DF(x))^{k-1}DG(x)\right)\left(Ax + AF(x) + R(x+ F(x))\right)\\
& & -DG(x)(AF(x)+R(x+F(x))).
\end{eqnarray*}
Therefore, the coefficient in front of $DG(x)$ has order $\geq 2=p_{j,1}$, for any $1\leq j\leq n$. Therefore, the analytic differential action is regular.

If $A$ is a diagonal matrix, $A = diag(\lambda _1,..., \lambda _n)$ and $Q = (q_1,...,q_n)\in \Bbb N^n$ is a multiindex, one has
$$\mathcal S\left(x^Q\frac{\partial }{\partial x_j}\right) = (\lambda _j - (Q , \lambda ))x^Q\frac{\partial }{\partial x_j}.$$
Here we have written $(Q,\la)=\sum_{i=1}^nq_i\la_i$.
It follows that the simplest supplementary space to
the image of $\mathcal S^{(i)}$, $i\geq 2$ is the vector space $\cR^{(i)}$ of {\bf resonant vector fields}:
it is generated by
$$
x^Q\frac{\partial }{\partial x_j},\quad \lambda_j = (Q , \lambda ),\, 1\leq j\leq n,\,|Q|=i. 
$$
This supplementary space $\cR^{(i)}$ together with Proposition \ref{prop-formal-nf} give the classical formal
Poincar\'e-Dulac normal form. Let $\pi$ be the projection nonresonant terms (i.e the image of $\cS$)~: 
$$
\pi\left(\sum_{j=1}^n\sum_{Q,\in\Bbb N^n, |Q|\geq 2 }f_{Q,j}x^Q\frac{\partial }{\partial x_j}\right)= \sum_{j=1}^n\sum_{{\scriptstyle (Q,\lambda)\neq \lambda_j  \atop
\scriptstyle Q \in \Bbb N^n, |Q|\geq 2 }}f_{Q,j}x^Q\frac{\partial }{\partial x_j}.
$$
Assume that the tuple $(\lambda _1,...,\lambda _n)$ belongs to the
{\bf Poincar\'e domain}. It means that the eigenvalues $\la_i$ lie on one side of a line through the origin of the complex plane (the line is excluded). Then, it is classical that one has the following inequalities 
\begin{equation}
\label{big-denom-vf}
\vert \lambda _j - (Q , \lambda )\vert \ge C\vert Q \vert , \ \ j=1,...,n, \ \ \text{for sufficiently large} \ \ \vert Q \vert
\end{equation}
where $C>0$ is a constant which does not depend on $Q$ and $j$. In that case, the {\bf big denominators property of order $1$} holds. Indeed, we have
$$
\cS^{-1}\circ \pi\left(\sum_{j=1}^n\sum_{Q,\in\Bbb N^n, |Q|=i\geq 2 }f_{Q,j}x^Q\frac{\partial }{\partial x_j}\right)= \sum_{j=1}^n\sum_{{\scriptstyle (Q,\lambda)\neq \lambda_j  \atop
\scriptstyle Q \in \Bbb N^n, |Q|=i\geq 2 }}\frac{f_{Q,j}}{\lambda _j - (Q , \lambda )}x^Q\frac{\partial }{\partial x_j}.
$$
Hence, we have 
$$
\|\cS^{-1}\circ \pi(f^{(i)})\|\leq \frac{\|\pi(f^{(i)})\|}{Ci}\leq \frac{\|f^{(i)}\|}{Ci}
$$
Therefore Theorem \ref{main1} implies the following classical result.

\begin{thm}[Poincar\'e, see \cite{Arn2}]
\label{thm-Poincare}
If the eigenvalues of $A$ belong to the Poincar\'e domain then the Poincar\'e-Dulac normal form holds in the
analytic category.
\end{thm}

Our results of section \ref{sec-Gevrey} imply a stronger theorem for the case that
instead of (\ref{big-denom-vf}) we have, for some $\alpha > 0$ the estimate
\begin{equation}
\label{Gevrey-big-denom-vf}
\vert \lambda _j - (Q, \lambda )\vert \ge C\vert Q \vert ^{1-\alpha}, \ \ j=1,...,n, \ \ \text{for sufficiently large} \ \ \vert Q \vert
\end{equation}
In this case our Theorem
\ref{thm-Gevrey} implies the following statement.

\begin{thm}
\label{thm-vf-Gevrey}
If the eigenvalues satisfy (\ref{Gevrey-big-denom-vf}) for some fixed $\alpha > 0$ (with a constant $C$
that does not depend on $Q$ and $j)$ then the Poincar\'e-Dulac normal form holds in the $\alpha$-Gevrey category~: for each analytic nonlinear perturbation of the linear vector field $S=\sum_{j=1}^n\la_ix_i\frac{\partial}{\partial x_i}$ is conjugate to a formal normal form by the mean of formal $\alpha$-Gevrey transformation.
\end{thm}

The particular case for which $0<\alpha <1$, is a recent result of P. Bonckaert and P. De Measschalck in \cite{bonckaert-demaesschalck}. In this case, according to the definition of the present paper, we have
relatively big denominators of order $1-\alpha <1$. But our Theorem \ref{thm-Gevrey} holds for $\alpha \ge 1$ as well, i.e. in the case that we ``do not have denominators" ($\alpha =1)$ or have small denominators ($\alpha >1)$.
Note that condition (\ref{Gevrey-big-denom-vf}) with $\alpha = 1+\tau, \tau >0$ is exactly the
Siegel condition of type $\tau $. Therefore in the case $\alpha >1$ we obtain, as a direct corollary of
Theorem \ref{thm-vf-Gevrey}, that if the eigenvalues satisfy the
Siegel condition of type $\tau >0$~: there exists a constant $C>0$ such that for all $Q\in\Bbb N^n$, with $|Q|\geq 2$
$$
0\neq |(Q,\lambda)-\lambda_i|\geq \frac{C}{|Q|^{\tau}}
$$
then the resonant normal form holds in the $(1+\tau )$-Gevrey category.
The latter result is also known, it was obtained by G. Iooss, E. Lombardi and L.Stolovitch in the works \cite{lombardi-nf} and \cite{stolo-lombardi}. In a more restricted situation, it is known that one can find ``holomorphic sectorial normalization'' with Gevrey asymptotic expansion \cite{stolo-boele}.

\subsection{Normal form of (non isolated) singularities}

Let $\cO_n$ be the space of germs of holomorphic functions at the origin of $\Bbb C^n$.
It is well know that a germ of analytic function $f$ at the origin of $\Bbb C^n$, having an isolated singularity there (i.e. $Df(0)=0$ and $0$ is isolated among the $p$'s such that $Df(p)=0$) is analytically conjugated to a polynomial $P$. This means that there exists a germ of analytic diffeomorphism of $(\Bbb C^n,0)$ such that $f\circ \Phi=P$. This has been extensively studied by V.I. Arnold and his school \cite{arnold-sing1,arnold-sing, Encyclo1}. The usual proof goes as follows~: first of all, since the singularity is isolated then the vector space ${\cO}_n/J_f$ is a finite dimensional vector space. Here, $J_f=(\frac{\partial f}{\partial x_1},\ldots, \frac{\partial f}{\partial x_n})$ denotes the {\it Jacobian ideal}, the ideal in  ${\cO}_n$ generated by the partial derivative of $f$. Then, according to Tougeron's theorem \cite{Tougeron}, \cite{arnold-sing1}[section 6.3], the Jacobian ideal contains a certain power of the maximal ideal ${\cM}_n$ : ${\cM}_n^{k}\subset J_f$. Then, using an homotopy method, one shows that, for any $r_{k}\in {\cM}_n^{k}$ the exists a family $\{\Phi_t\}_{t\in [0,1]}$ of diffeomorphisms of $(\Bbb C^n,0)$ such that
$\Phi_t^*(f+tr_{k})=f$. Then, we obtain the desired result if we set $r_k:=-(f-j^{k-1}(f))$.

\medskip

What does happen when the singularity is {\it a priori} not isolated ? We shall consider the case where $f$ is a perturbation of a homogeneous polynomial $f_0$ of higher order. Both of them are supposed not to have an isolated singularity at the origin. In that case, the vector space ${\cO}_n/J_{f_0}$ is not finite dimensional. Nevertheless, we shall prove that $f=f_0+R$ is formally conjugate to a formal normal form. It's, {\it a priori}, a formal power series satisfying some specific conditions to be defined below. If this normal form was holomorphic in a neighborhood of the origin (that's the case for a polynomial normal form for instance), then we could conclude, by Artin's theorem, that there actually exists an holomorphic transformation to that normal form. The main problem in the general situation is that there no reason why $f$ should have {\it a priori} an holomorphic normal form (it is the case for isolated singularity).

\medskip

In the following, we shall first give a definition of a normal form of a perturbation of $f_0$ with respect to $f_0$. Then we shall prove that there {\it exists an analytic change of coordinates to such a normal form}.
Let us first recall the division theory developed by H. Grauert, H. Hauser and A. Galligo.

\subsubsection{Division theorem}
Let us consider $\al_0,\al_1,\ldots,\al_n$ be nonnegative real numbers which are linearly independent over $\Bbb Q$. Let us consider the affine
 linear form $L(v):=\al_0(v_0-1)+\sum_{i=1}^n\al_iv_i$ on $\Bbb R^{1+n}$. Let $f=\sum_{Q\in \Bbb N^n}f_Qx^Q\in \left(\Bbb C[[x_1,\ldots ,x_n]]\right)^q$ a formal map. We shall denote $x^Q\partial_i:= (0,\ldots,0,x^Q,0,\ldots, 0)$ where $x^Q$ is $i$th component of the vector. Hence, $f$ can be written as $f=\sum_{Q\in \Bbb N^n,i\in \{1,\ldots, q\}}f_{i,Q}x^Q\partial_i$. With this notation, the $i$th component reads $f_i:= \sum_{Q\in \Bbb N^n}f_{i,Q}x^Q$. The {\it initial part} of $f$ is defined to be $In(f):=f_{i_0,Q_0}x^{Q_0}\partial_{i_0}$ where $(Q_0,i_0)\in \Bbb N^n\times\{1,\ldots,q\}$ is the unique minimum $L((i_0,Q_0))=\min_{(i,Q)\in Supp(f)}L(Q)$ and $f_{i_0,Q_0}\neq 0$.

\medskip

Let us define for $s>0$ sufficiently small, $|f|_s:=\sum_{Q\in\Bbb N^n}|f_Q|s^{L(1,Q)}$ where $f=\sum_{Q\in\Bbb N^n}f_Qx^Q\in \cO_n$. If $f=\sum_{Q\in \Bbb N^n,i\in \{1,\ldots, q\}}f_{i,Q}x^Q\partial_i\in \cO_n^q$, then we shall write $|f|_s:=\sum_{i=1}^qs^{L(i,0)}|f_i|_s$. Consider the $\cO_n$-submodule $I$ of $\cO_n^q$ generated by the germs of holomorphic maps $f_1,\ldots, f_r$. Let us define the {\it initial module} $In(I)$ of $I$ to be $\cO_n$-submodule $I$ of $\cO_n^q$ generated by $In(f_1),\ldots, In(f_r)$. Let us define
$$
\De(I):=\{g\in\cO_n\,|\,\text{no monomial of the Taylor expansion at $0$ belongs to } In(I)\}.
$$
Let $m_1,\ldots, m_p$ be a standard basis of $I$ with initial terms $\mu_1,\ldots, \mu_p$ respectively. Let us split the support of $In(I)$ into a disjoint union $\bigcup_{i=1}^rM_i$ of sets $M_i$ such that $M_i\subset \text{supp }(\cO_n\mu_i)$. Let us define
$$
\nabla(I):=\{a=(a_1,\ldots,a_p)\in\cO_n^p\,|\,\text{supp }(a_i\mu_i) \subset M_i, \; i=1,\ldots, p\}.
$$

\begin{thm}\label{galligo-hauser}\cite{galligo-fourier,hauser-muller}
Let $l:\cO_n^p\rightarrow \cO_n^q$ be the $\cO_n$-linear map defined by $l(a)=\sum_{i=1}^pa_im_i$. Assume that the $m_i$'s form a standard basis of $I:=\text{Im }(l)$. Let $K:=\text{Ker }l$. Then the following holds:
\begin{enumerate}

\item \begin{equation}\label{decompo-fond}
\cO_n^q=I\oplus \De(I),\quad \cO_n^p=K\oplus \nabla(I);
\end{equation}
\item There is a constant $c>0$, such that for all small enough $s>0$ :  for each $e\in \cO_n^q$ with $|e|_s<+\infty$, then there exists a unique $a\in\nabla(I)$ and $b\in \De(I)$ with $|a|_s,|b|_s<+\infty$ such that $e=l(a)+b=\sum_{i=1}^p a_im_i +b$ and
$$
(\min_i|m_i|_s)|a|_s+|b|_s\leq c|e|_s.
$$
\end{enumerate}
\end{thm}

For $q=1$, $I$ is just an ideal on $\cO_n$ and the decomposition $(\ref{decompo-fond})$ reads
$
\cO_n=I\oplus \De(I).
$
The previous theorem asserts that, for any $g\in \cO_n$, there exists $a:=(a_1,\ldots,a_p)\in \nabla(I)$ and $h\in \De(I)$, such that
$
g=\sum_{i=1}^pa_im_i +h
$
and there exists constant $c_r$ (independent of $g$) such that
$
|h|_r<c_r|g|_r,\quad |a|_r<c_r|g|_r.
$.

\begin{remark}
If $e$ vanishes up to order $k$ at the origin then each $a_i$ vanishes up to order $k-k_i$, where $k_i$ is the order of $m_i$ at the origin. The remainder $b$ vanishes up to order $k$.
\end{remark}
\subsubsection{Normal form of deformations of a homogeneous polynomial}


Let $f_0$ be a homogeneous polynomial of degree $q\geq 2$. Let us consider an holomorphic perturbation $f$ of $f_0$ of higher order~: $R=f-f_0$ is a germ of holomorphic function of order $\geq s+1$ at the origin. Let $I:=J_{\nabla f_0}$ be the Jacobian ideal of $f_0$. With the notation above, we set $r:=n$ and $f_i:=\frac{\partial f_0}{\partial x_i}$. Let $m_1,\ldots, m_p$ be a standard basis of $I$. We can write $m_j=\sum_{k=1}^n m_{j,k}f_k$ for some analytic germs $m_{j,k}$ at the origin.

\medskip

Let us define the {\it cohomological operator} ${\mathcal S}:\cO_n^n\rightarrow \cO_n$ to be $${\mathcal S}(U):=Df_0.U.$$
Let $U$ be a germ of holomorphic vector field of positive order $k\geq 2$. Let us consider conjugacy of $f$ with respect to the diffeomorphism $\Phi=id+U$~: we have
\begin{eqnarray}
f(id+U)&= &f_0+ Df_0.U+R+ \underbrace{f_0(id+U)-f_0-Df_0.U}_{=:\Sig_1}\label{conj-equ}\\
&& + \underbrace{R(id+U)-R}_{=:\Sig_2}.\nonumber
\end{eqnarray}
The orders of $Df_0.U$ and $R$ are are greater or equal than $k+s-1$ and $q+1$ respectively.
On the other hand, the orders of $\Sig_1$ and $\Sig_2$ are greater or equal than $2k+q-2$ and $k+q$ respectively. Both are $\geq q+2$.
\begin{prop}[Formal normal form of singularity]
There exists a formal change of coordinates $\widehat \Phi$ such that
$$
f\circ\widehat \Phi-f_0\in \De(I)\otimes \widehat\cO_n
$$
\end{prop}
\begin{proof}
Indeed, we construct by induction on the order $l$, a formal vector field $U=\sum_{l \geq 2}U_l$ such that
$$
Df_0.U+R+\Sig_1+\Sig_2\in \De(I),
$$
Hence, the homogeneous part of degree $i\geq q+1$ of the Taylor expansion at $0$ reads
$$
Df_0.U_{i-q+1}+R^{(i)}+\{\Sig_1+\Sig_2\}^{(i)} \in \De(I)
$$
where $\{\Sig_1+\Sig_2\}^{(i)} $ denotes the homogeneous of degree $i$ of $\Sig_1+\Sig_2$. It is a polynomial in the $U_j$'s, $j<i-q+1$. Let us decompose $R_i+\{\Sig_1+\Sig_2\}_i$ along $(\ref{decompo-fond})$~: there exists $U_{i-q+1}$ and $h^{(i)}\in \De(I)$ such that $Df_0U_{i-q+1}-h^{(i)}= -R^{(i)}-\{\Sig_1+\Sig_2\}^{(i)} $, that is $Df_0U_{i-q+1}+R^{(i)}+\{\Sig_1+\Sig_2\}^{(i)} =h^{(i)}\in \De(I)$.
\end{proof}
Now we apply the variation of our main theorem, Theorem  \ref{var-main-thm} with $m=0$, $d=1$ to prove
\begin{thm}
\label{thm-functions}
Let $f_0$ be a homogeneous polynomial of degree $q$. Let $f=f_0+R_{>q}$ be an analytic deformation of $f_0$ by $R_{>q}$ which is a order greater than $q$. Then there exists an analytic change of coordinates $\Phi$ in a neighborhood of the origin of $\Bbb C^n$ such that $f\circ \Phi -f_0\in \De(I).$
\end{thm}

\begin{proof}
Set $\cS U:=Df_0U$. We have
$$
Df_0U=\sum_{i=1}^pa_im_i=\sum_{i=1}^pa_i\left(\sum_{k=1}^nm_{i,k}f_k\right)=\sum_{k=1}^n\left(\sum_{i=1}^pa_im_{i,k}\right)f_k
$$
According to theorem \ref{galligo-hauser}, the problem to solve has the big denominator property of order $0$. Indeed, we obtain a good estimate of the $a_i$'s in term of $Df_0U$; so, we obtain good estimate of $U_i:=\sum_{i=1}^pa_im_{i,k}$ in term of $Df_0U$. Moreover, according to the second point of that theorem, the projection $\pi$ onto the image $J_{f_0}$ satisfies
$$
|\pi(f)|_r\leq  \left|\sum_{i=1}^n a_i\frac{\partial f_0}{\partial x_i}\right|_r\leq c_r |f|_r.
$$
where $f=\sum_{i=1}^n a_i\frac{\partial f_0}{\partial x_i}+h$ and $h\in \De(I)$. 
Furthermore, since $R$ has order $\geq q+1$ at the origin, then $\frac{\partial R}{\partial x_j}$ has order $\geq q$ at the origin. That is,the associated operator $\cT$ is $1$-regular. Therefore Theorem \ref{thm-functions}
is a direct corollary of Theorem \ref{galligo-hauser} and our Theorem \ref{var-main-thm} with $m=0$, $q=s$ and $d=1$.
\end{proof}

\begin{remark}
In the case of isolated singularity, our proof of conjugacy to a polynomial normal form is a direct one. Theorem \ref{galligo-hauser} shows that $\De(J_{f_0})$ contains only a finite number of monomials. This replaces Tougeron's theorem. Our method replaces the usual ``homotopic method''. Our proof is also quite different from Arnold's original one \cite{arnold-sing} which is a kind a Newton method.
\end{remark}
\begin{remark}
In the non-isolated case and if the perturbation is formally conjugated to an analytic normal form $g$ then, Artin's theorem \cite{artin} shows that there exists a germ of analytic diffeomorphism that conjugates $f$ to $g$. Our result shows that the same holds if $g$ is only a formal normal form (and such a formal conjugacy always exists) without using the difficult theorem of Artin.
 \end{remark}

\begin{appendices}
\section{Normal form for $n$-tuples of linearly independent vector fields on $\mathbb R^n$, Riemannian metrics and conformal structures by M. Zhitomirskii}\label{conformal-section}

\subsection{Analytic normal forms}

A Riemannian metrics on $\mathbb R^n$ can be treated as an $n$-tuple of pointwise linearly
independent vector fields on $\mathbb R^n$ defined up to multiplication by an $n\times n$ matrix
$T(x)\in SO(n)$.  A conformal structure on $\mathbb R^n$ can be treated as an $n$-tuple of pointwise linearly
independent vector fields on $\mathbb R^n$ defined up to multiplication by an $n\times n$ matrix
$T(x)\in SO(n)$ and by a non-vanishing function $H(x)$.

It is convenient to associate to each of the object in the title of this section an $n\times n$ matrix
whose entries are analytic function germs.

\begin{definition}
\label{associated-matrix}
Given an $n$-tuple of vector fields
$$V_i = f_{1i}(x)\frac{\partial}{\partial x_1} + \cdots +  f_{ni}(x)\frac{\partial}{\partial x_n}, \quad i=1,\ldots, n$$
we associate to it an $n\times n$ matrix $M(x)$ in which the $k$th {\it column} is the tuple $\left(f_{1k}(x), ..., f_{nk}(x)\right)^t$ of coefficients of the vector field $V_k$.
Given a Riemannian metrics, respectively a conformal structure on $\mathbb R^n$,
we treat it as a tuple of pointwise linearly
independent vector fields $V_1,..., V_n$ on $\mathbb R^n$ defined up to multiplication by an $n\times n$ matrix
$T(x)\in SO(n)$, respectively up to multiplication by $T(x)\in SO(n)$ and by a scalar function $H(x)$, and we associate to the Riemannian metrics, respectively the conformal structure, the matrix $M(x)$ associated with $(V_1,...,V_n)$.
\end{definition}

\begin{thm}
\label{thm-anal-nf-3-objects}
An $n$-tuple of pointwise linearly independent analytic vector field germs on $\mathbb R^n$ can be reduced by a local
analytic diffeomorphism to a normal form with the associated matrix $M(x)$ satisfying the equation
\begin{equation}
\label{eq-nf-tuple}
M(x)\cdot \begin{pmatrix}x_1\cr \vdots \cr x_n\end{pmatrix} \equiv 0.
\end{equation}
A germ of analytic Riemannian metrics on $\mathbb R^n$ can be reduced by a local analytic diffeomorphism to a normal form with the
associated matrix $M(x)$ satisfying  (\ref{eq-nf-tuple}) and the equation
\begin{equation}
\label{eq-nf-Riem}
M^t(x) \equiv M(x).
\end{equation}
A germ of analytic conformal structure on $\mathbb R^n$ can be reduced by a local analytic diffeomorphism to a normal form with
the associated matrix $M(x)$ satisfying (\ref{eq-nf-tuple}), (\ref{eq-nf-Riem})  and the equation
\begin{equation}
\label{eq-nf-conformal}
\text{trace }M(x) \equiv 0.
\end{equation}
\end{thm}

For $n=2$ the given normal form for Riemannian metrics is close (though not the same) to the Gauss lemma on a certain property of geodesic coordinates, see \cite{Gauss}. The Gauss lemma can be generalized to any $n$. Its proof, in formal category and for any $n$, in terms of normal forms, can be found in \cite{Spain}.

\medskip

Note that any matrix $n\times n$ matrix $M(x)$ satisfying (\ref{eq-nf-tuple}) and
(\ref{eq-nf-Riem}) starts with terms of order $\ge 2$. Its quadratic part {\it can be identified with the space of
Riemannian curvature  tensors}. If $n=2$ then such a quadratic part be written as
\begin{equation*}
M^{(2)}(x) = K\cdot \begin{pmatrix}x_2^2 & -x_1x_2\cr -x_1x_2 & x_1^2\end{pmatrix}, \ \ K\in \mathbb R
\end{equation*}
and the parameter $K$ {\it can be identified with the curvature of Gauss.}

\medskip

Note also that any $2\times 2$ matrix $M(x)$ satisfying (\ref{eq-nf-tuple}),
(\ref{eq-nf-Riem}), (\ref{eq-nf-conformal}) is the zero matrix.
This matches the well
known theorem that for $n=2$ any conformal structure is locally conformally flat, see \cite{conf-str}.

\medskip

For $n=3$,
respectively $n\ge 4$, any $n\times n$ matrix $M(x)$  satisfying (\ref{eq-nf-tuple}),
(\ref{eq-nf-Riem}), (\ref{eq-nf-conformal}) starts with terms of order $\ge 3$, respectively of order $\ge 2$.  If $n=3$ then the terms of order $3$ can be identified with the Cotton tensor, and if $n\ge 4$ then the terms of order $4$ can be identified with the Weyl tensor.

\medskip

Certainly these relations between the normal form and the classical tensors do {\it not} require
analytic normal form. For these relations it is enough to have the given normal form in formal category
and moreover, a normal form for a jet of small order will be enough.

\medskip

In the following subsections we will explain:

\medskip

\noindent  (a) \ how the normal forms of Theorem \ref{thm-anal-nf-3-objects} were obtained in the
formal category;

\medskip

\noindent (b) \ how Theorem \ref{thm-anal-nf-3-objects} follows from the same results in the formal category
and our Theorem \ref{main1}, i.e. we will prove that in the three classification problems of this section
there are big denominators.

\subsection{Explanation of formal normal forms. Belitskii inner product}
\label{sec-Bel-inner-product}

We will use the following notations:

\medskip

\noindent $\bullet $ \ By $\an(\mathcal M_{r\times r})$
we denote the space of $r\times r$ matrices whose entries are germs at $0\in \mathbb R^n$ of analytic
functions of $n$ variables.

\medskip


\noindent $\bullet $ By $\an^{(i)}(\mathcal M_{r\times r})$ we denote the subspace of $\an(\mathcal M_{r\times s})$ consisting of
matrices whose entries are homogeneous functions of degree $i$.

\medskip

\noindent $\bullet $ By $\an(\mathcal M_{r\times r})_{>d}$ we denote the subspace of $\an(\mathcal M_{r\times r})$ consisting of
matrices whose entries have zero $d$-jet at $0$.


\medskip


\noindent $\bullet $ \ $\an(\text{so}(n)) = \left\{M\in \an(\mathcal M_{r\times r}): \ M(x) = - M^t(x)\right\}$,
i.e. the space of skew-symmetric matrices whose entries are analytic function germs.

\medskip

\noindent $\bullet $ \ Finally we will use notation  $\mathbf x = \begin{pmatrix}x_1\cr \vdots \cr x_n\end{pmatrix}.$

\medskip

For the problem of local classification of $n$-tuples of vector field germs on $\mathbb R^n$ of the form
\begin{equation}
\label{n-tuples-vf}
V_1 = \pxone + h.o.t. , \ \cdots , \ V_n = \frac{\partial}{\partial x_n} + h.o.t.
\end{equation}
with respect to local diffeomorphisms one has, in terms of notations of section \ref{sec-big-denom-classif-problems},
$$r = n, \  m_i = 1, \ 1\leq i\leq n, \ s = n^2, \ \ q = 0, \ \mathcal F_{1,\bf m} = \left(\mathbb A_n^n\right)_{>1}.$$
The group $\mathcal G$ acts on the affine space $I+ \an(\mathcal M_{n\times n})_{>0}$. Here, $I$ denotes the constant matrix $I=\text{diag}(1,\ldots, 1)$. The action is as follows $$(id + \phi (x))_*(I + M(x)) = \left(I + D\phi(x)\right)^{-1}\left(I + M(x+\phi (x)\right).$$
It is a differential action of order $1$.
The linear operator $\mathcal S$ is the operator

\begin{equation}
\label{S-for-tuples}
\mathcal S_1: \ \left(\mathbb A_n^n\right)_{>1} \ \to \ \an(\mathcal M_{n\times n})_{>0}, \ \
\mathcal S_1(\phi ) = - D\phi(x).
\end{equation}
Hence, we define
\begin{eqnarray*}
\cT_1(M;\phi) &:=&\left[\left(I + D\phi(x)\right)^{-1}-I+D\phi(x)\right]\left(I + M(x+\phi (x)\right)\\
&&-D\phi(x)M(x+\phi (x))+M(x+\phi (x)).
\end{eqnarray*}
Therefore, we have
$$
(id + \phi (x))_*(I + M(x))=I+\cS_1(\phi)+\cT_1(M;\phi)
$$
We have $\cT(M, 0)$ has order $\geq 1$ at the origin. Let us investigate the regularity of $\cT_1$. We have,
\begin{eqnarray*}
\frac{\partial \cT_1}{\partial \phi}(\phi)G (x) &=& \left[\left(I + D\phi(x)\right)^{-1}-I+D\phi(x)\right]DM(x+\phi(x))G\\
&&+D\phi(x)DM(x+\phi(x))G(x) +DM(x+\phi(x))G(x).
\end{eqnarray*}
Since $M$ vanishes at the origin, the coefficient in front of $G$ has order $\geq 0=p_{j,0}$. On the other hand, we have
\begin{eqnarray*}
\frac{\partial \cT_1}{\partial \phi'}(\phi)DG (x) &=& \left[\sum_{k\geq 2}(-1)^kk[D\phi(x)]^{k-1}DG\right]\left(I + M(x+\phi (x)\right)\\
&&+DG(x)M(x+\phi(x)).
\end{eqnarray*}
Since $\phi$ has order $\geq 2$ at the origin, then $D\phi(x)$ has order $\geq 1$. Moreover, $M$ has order $\geq 1$ at the origin. Hence the coefficient in front of $DG$ has order $\geq 1=p_{j,1}$, for all $1\leq j\leq n$. Therefore, the differential analytic map $\cT_1$ is regular.

To find a complementary space to the image of $\mathcal S^{(i)}_1$, the restriction of $\mathcal S_1$ to
$\left(\mathbb A_n^n\right)^{(i)}$ we use the following inner product introduced by G. Belitskii
(in fact, it goes back to \cite{fischer}) and used by G. Belitskii
him to construct a number of formal normal forms in various local classification problems
\cite{Bel-1}.

\begin{definition}
The Belitskii inner product of
two homogeneous functions of $n$ variables of the same degree $i$,  is defined as follows:
$$f = \sum_{\vert \alpha \vert =i}f_\alpha x^\alpha , \
g = \sum_{\vert \alpha \vert =i}g_\alpha x^\alpha  \ \implies
<f,g> = \sum _{\vert \alpha \vert = i}\alpha ! f_\alpha \bar g_\alpha .$$
Here $\alpha = (\alpha _1,..., \alpha _n)$ is a multiindex,
$x^\alpha = x_1^{\alpha _1}\cdots x_n^{\alpha _n}$ and $\alpha ! = \alpha _1!\cdots \alpha _n!$. The inner product of tuples of functions $f,g\in \left(\mathbb A_n^k\right) ^{(i)}, \ f = (f_1,...,f_k), g = (g
_1,...,g_k)$
is defined by $<f,g> = <f_1,g_1> + \cdots + <f_k,g_k>$.
\end{definition}

\ni This inner product is very convenient because of the following statement.

\begin{prop}[\cite{Bel-1}]
\label{prop-Bel-inner-product}
The operator $f\to \frac{\partial f}{\partial x_j}$ has conjugate operator, with respect to the Belitskii  inner product, $g\to x_jg$.
\end{prop}

As an immediate corollary of this proposition we obtain that the adjoint operator to
the operator (\ref{S-for-tuples}) with respect to Belitskii inner product has the form $M(x)\to -M(x)\cdot \mathbf x$. Therefore the
orthogonal complement to the image of $\mathcal S^{(i)}_1$ consists of matrices $M(x)$ such that
$M(x)\cdot \mathbf x = 0$. Now the first normal form of Theorem \ref{thm-anal-nf-3-objects} holds in the formal category by Proposition \ref{prop-formal-nf}.

\medskip

Now, for the problem of local classification of {\it Riemannian metrics} on $\mathbb R^n$, viewed
 as the problem as
classification the $n$-tuples of vector fields of form (\ref{n-tuples-vf}) with respect to local diffeomorphisms and multiplication by a matrix-function $exp Q(x), Q(x)\in \text{so}(n)(x)$, one has, in terms of notations of section \ref{sec-big-denom-classif-problems}
$$
r= n+n(n-1)/2,\ m_i=1,\ 1\leq i\leq n,\ \ m_i=0, \ n+1\leq i\leq r
$$
$$s = n^2, \ \ q = 0, \
\mathcal F = \left(\mathbb A_n^n\right)_{>1}\times \an(\text{so}(n))_{>0}.$$
The group $\mathcal G = id + \mathcal F$ acts
the affine space $I+ \an(\mathcal M_{n\times n})_{>0}$ as follows:
$$\left((id + \phi (x), exp Q(x)\right)_*(I + M(x)) = exp Q(x)\cdot \left(I + D\phi(x)\right)^{-1}\left(I + M(x+\phi (x)\right).$$
The linear operator $\mathcal S$ is the operator
\begin{equation}
\label{S-for-Riem}
\begin{split}
\mathcal S: \ \left(\mathbb A_n^n\right)_{>1}\times \an(\text{so}(n))_{>0} \ \to   \ \an(\mathcal M_{n\times n})_{>0}, \\
{\mathcal S}_2(\phi (x), Q(x)) = - D\phi(x) + Q(x)
\end{split}
\end{equation}
Let us define
$$
\cT_2(M;\phi,Q):= (I+Q)\cT_1(M;\phi)+Q\cS_1(\phi)+\left(\sum_{k\geq 2}\frac{Q^k}{k!}\right)(I+\cS_1(\phi)+\cT_1(M;\phi)).
$$
We have $\text{ord}_0(\cT_2(M;0))\geq 1$. From the previous computation and the fact that $Q$ has order $\geq 1$ at the origin, we obtain
\begin{eqnarray*}
\text{ord}_0\left(\frac{\partial \cT_2}{\partial \phi}\right) & \geq & 0=p_{j,0},\quad 1\leq j\leq n\\
\text{ord}_0\left(\frac{\partial \cT_2}{\partial \phi'}\right) & \geq & 1=p_{j,1},\quad 1\leq j\leq n\\
\text{ord}_0\left(\frac{\partial \cT_2}{\partial Q}\right) & \geq & 1=p_{j,0},\quad ,\quad n+1\leq j\leq r.
\end{eqnarray*}
Therefore, $\cT_2$ is regular.
Using Proposition \ref{prop-Bel-inner-product} it is easy to prove that the adjoint operator to
(\ref{S-for-Riem}) with respect to Belitskii inner product has the form
$$M(x) \ \to \left(-M(x)\cdot \mathbf x, \frac{1}{2}\left( M(x)-M^t(x)\right)\right).$$
Indeed, since we have $Q+Q^t=0$, then $\left<Q,M\right>=-\left<Q,M^t\right>$ for any matrix germ. As a consequence, we have
\begin{eqnarray*}
\left<{\mathcal S}_2(\phi (x), Q(x)), M\right> &=& \left<- D\phi(x) + Q(x), M\right>\\
& = & \left<\phi, -M.\bf x\right> +\left<Q,M\right>\\
& = & \left<\phi, -M.\bf x\right> +\left<Q,\frac{1}{2}(M-M^t)\right>
\end{eqnarray*}
As a consequence the orthogonal complementary space with respect to the Belitskii inner product gives
the normal form of Theorem \ref{thm-anal-nf-3-objects} for Riemannian metrics in the formal category.

\medskip

In the same way we obtain the normal form of Theorem \ref{thm-anal-nf-3-objects} for conformal structures.
In this case we have
$$
r=n+ n(n-1)/2+1,\ m_i=1,\ 1\leq i\leq n,\ m_i=0,\ n+1\leq i\leq r.
$$
$$s = n^2, \ q = 0, \
\mathcal F = \left(\mathbb A_n^n\right)_{>1}\times \an(\text{so}(n))_{>0}\times \left(\mathbb A_n^1\right)_{>0}.$$
The action is define to be
$$
\left((id + \phi (x), exp Q(x),1+h(x)\right)_*(I + M(x)) = (1+h(x))exp Q(x)\cdot \left(I + D\phi (x)\right)^{-1}\left(I + M(x+\phi (x)\right).
$$
and the operator $\mathcal S_3$ has the form
\begin{equation}
\label{S-for-Riem}
\begin{split}
\mathcal S_3: \ \left(\mathbb A_n^n\right)_{>1}\times \an(\text{so}(n))_{>0}\times \left(\mathbb A_n^1\right)_{>0} \ \to   \ \an(\mathcal M_{n\times n})_{>0}, \\
{\mathcal S}(\phi (x), Q(x), h(x)) = - D\phi(x) + Q(x) + h(x)\cdot I.
\end{split}
\end{equation}
We define the associates differential analytic map of order $(1,\ldots,1,0,\ldots,0, 0)$~:
$$
\cT_3(M;\phi,Q,f):= \cT_2(M;\phi,Q)+h\cS_2(\phi,Q)+h\cT_2(M;\phi,Q). 
$$
We have $\text{ord}_0(\cT_3(M;0))\geq 1$. Moreover, we have
\begin{eqnarray*}
\text{ord}_0\left(\frac{\partial \cT_3}{\partial \phi}\right) & \geq & 0=p_{i,0},\quad 1\leq i\leq n\\
\text{ord}_0\left(\frac{\partial \cT_3}{\partial \phi'}\right) & \geq & 1=p_{i,1},\quad 1\leq i\leq n\\
\text{ord}_0\left(\frac{\partial \cT_3}{\partial Q}\right) & \geq & 1=p_{i,0},\quad n+1\leq i\leq r-1\\
\text{ord}_0\left(\frac{\partial \cT_3}{\partial h}\right) & \geq & 1=p_{r,0}.
\end{eqnarray*}
The conjugate linear operator with respect to the Belitskii inner product has the form
$$M(x) \ \to \left(-M(x)\cdot \mathbf x, \ \frac{1}{2}(M(x)-M^t(x)), \ \text{trace }M(x)\right)$$
which implies the last normal form of  Theorem \ref{thm-anal-nf-3-objects} in the formal category.

\subsection{Big denominators. Proof of Theorem \ref{thm-anal-nf-3-objects}}

The action of the group in each of the classification problems of this section is a differential
action of order

\noindent ${\bf m}=(1,\ldots,1)\in\Bbb N^n$ for tuples of vector fields;

\smallskip

\noindent ${\bf m} = (1,\ldots 1, 0,\ldots, 0)\in\Bbb N^{n+n(n-1)/2}$ for Riemannian metrics;

\smallskip

\noindent ${\bf m}= (1,\ldots 1, 0,\ldots, 0, 0)\in\Bbb N^{n+n(n-1)/2+1}$ for conformal structures.

\medskip

In this section we will prove that we have big denominators of order 1 and that the formal normal form
 in Theorem \ref{thm-anal-nf-3-objects} is uniformly bounded. In other words, we will prove that the assumptions
 of  Theorem \ref{main1} hold true and consequently Theorem \ref{thm-anal-nf-3-objects} holds not only in formal, but also in analytic category. Introduce
$$\mathcal N = \left\{A\in \an(\mathcal M_{n\times n}): \ \ A(x)\cdot \mathbf x \equiv 0\right\}$$
$$\mathcal N_{RM} = \left\{A\in \an(\mathcal M_{n\times n}): \ \ A(x)\cdot \mathbf x =0, \ A(x) \equiv  A^t(x)\right\}$$     $$\mathcal N_{CS} = \left\{A\in \an(\mathcal M_{n\times n}): \ \ A(x)\cdot \mathbf x\equiv 0, \ A(x)\equiv  A^t(x), \ \text{trace } A(x)\equiv 0\right\}.$$
To prove that that the assumptions
 of  Theorem \ref{main1} hold true we have to work with the equations
\begin{equation}
\label{eq-a-coh}
\begin{split}
A-
D\phi\in \mathcal N
\end{split}
\end{equation}
\begin{equation}
\label{eq-b-coh}
\begin{split}
A -D\phi + Q\in \mathcal N_{RM}, \ \ Q = -Q^t
\end{split}
\end{equation}
\begin{equation}
\label{eq-c-coh}
\begin{split}
A -D\phi + Q + h\cdot I\in \mathcal N_{CS}, \ Q = -Q^t
\end{split}
\end{equation}
with respect to
$\phi \in \an(\mathcal M_{n\times 1}), \ Q \in \an(\mathcal M_{n\times n}), \ h\in \an.$
The fact that the assumptions of Theorem \ref{main1} hold true follows from the following proposition.
\begin{prop}
\label{prop-estimates-three-objects}
Consider equations (\ref{eq-a-coh}), (\ref{eq-b-coh}), (\ref{eq-c-coh}) with $A\in \an^{(i)}(\mathcal M_{n\times n})$,
\ni \ $i\ge 2.$ In the case of equation (\ref{eq-c-coh}) assume that $n\ge 3$. Each of these equations has a unique
solution such that $D\phi\in  \an^{(i+1)}(\mathcal M_{n\times 1})$, \ $Q\in \an^{(i)}(\mathcal M_{n\times n})$, \
$h\in \an^{(i)}$
and for some $C>0$ which depends neither on $i$ nor on $A$  one has the estimates
$$\vert \vert \phi  \vert \vert < \frac{C}{i}\vert \vert A \vert \vert, \ \vert \vert R \vert \vert < C\vert \vert A \vert \vert , \ \vert \vert h \vert \vert < C\vert \vert A \vert \vert $$
where the norm $\vert \vert \cdot \vert \vert $ of a matrix whose entries are homogeneous functions is
the maximum of the norms of the entries and the norm of a homogeneous function is the sum of the absolute
values of its coefficients.
\end{prop}

The rest of the section is devoted to the proof of this proposition. The proof is very simple for
equation (\ref{eq-a-coh}) and more involved for (\ref{eq-b-coh}) and (\ref{eq-c-coh}), especially for
(\ref{eq-c-coh}). Throughout the proof we will use the Euler vector field
$$E = x_1\pxone + \cdots + x_n\pxn .$$

\ni {\bf Proof of Proposition \ref{prop-estimates-three-objects} for equation (\ref{eq-a-coh})}.
 This equation can be written in the form
$A(x)\cdot \mathbf x  - \phi ^\prime (x) \cdot \mathbf x \equiv 0$
or equivalently $E(\phi (x)) = A(x)\cdot \mathbf x$ ($E(\phi (x))$ denotes the vector $(E(\phi_i(x))_{i=1,\ldots, n}$).
Since $A(x)\in \mathcal A_{n\times n}^{(i)}(x)$ this equation has unique solution
$\phi (x) = \frac{1}{i+1}A(x)\cdot \mathbf x \in \an^{(i+1)}(\mathcal M_{n\times 1})$ which
implies Proposition \ref{prop-estimates-three-objects} for equation (\ref{eq-a-coh}).

\medskip

\ni {\bf Proof of Proposition \ref{prop-estimates-three-objects} for equation (\ref{eq-b-coh})}.
This equation can be expressed in the form

\begin{equation*}
\begin{split}
\phi ^\prime (x) - \left(\phi ^\prime (x)\right)^t - 2Q(x) = A(x) - A^t(x)\\
\left(\phi ^\prime (x) - Q(x)\right)\cdot \mathbf x = A(x)\cdot \mathbf x .
\end{split}
\end{equation*}

\ni We can exclude $Q(x)$ from the first equation

\begin{equation}
\label{eq-for-Q}
Q(x) = \frac{1}{2}\left(\phi ^\prime (x) - \left(\phi ^\prime (x)\right)^t - A(x) + A^t(x)\right).
\end{equation}

\ni The second equation takes the form

\begin{equation}
\label{eq-for-phi-1}
\left(\phi ^\prime (x) + \left(\phi ^\prime (x)\right)^t\right)\cdot \mathbf x  =
\left(A(x) + A^t(x)\right)\cdot \mathbf x.
\end{equation}



\ni Introduce the function
\begin{equation}
\label{eq-U}
U(x) = \ <\phi (x), \mathbf x> \ \ \text{and} \ \ \nabla  U(x) = \begin{pmatrix}\frac{\partial U}{\partial x_1},& \cdots &, \frac{\partial U}{\partial x_n}\end{pmatrix}^t
\end{equation}
where $< \cdot >$ denoted the standard inner product of two $n$-tuples.
Note that
$$\phi ^\prime (x)\cdot \mathbf x = E(\phi (x)), \ \ \left(\phi ^\prime (x)\right)^t\cdot \mathbf x  = \nabla U(x) - \phi (x).$$
Therefore equation (\ref{eq-for-phi-1}) can be expressed in the form
\begin{equation}
\label{eq-for-phi-3}
\begin{split}
E(\phi ) - \phi  = - \nabla U(x) + \left(A(x) + A^t(x)\right)\cdot \mathbf x.
\end{split}
\end{equation}
Take the inner product of this equation with the vector $\mathbf x$. We obtain
\begin{equation}
\label{eq-phi-a}
\left<E(\phi (x))) , \mathbf x\right> - U(x) = - \left<\nabla U(x) , \mathbf x\right> +  \left<\left(A(x)+A^t(x)\right)\cdot \mathbf x , \mathbf x \right>
\end{equation}
Note that $\left<E(\phi ) , \mathbf x \right> = E(U(x)) - U(x)$ and  $\left<\nabla U(x) , \mathbf x\right> = E(U(x))$.  Therefore equation (\ref{eq-phi-a}) gives the following equation for $U(x)$:
\begin{equation}
\label{eq-for-U}
E(U(x)) - U(x) = \frac{1}{2}\left<\left(A(x)+A^t(x)\right)\cdot \mathbf x , \mathbf x\right>.
\end{equation}
\ni Since $A\in \an^{(i)}(\mathcal M_{n\times n})$ it follows $\left<\left(A(x)+A^t(x)\right)\cdot \mathbf x , \mathbf x\right> \in \an^{(i+2)}(\mathcal M_{1\times 1})$. Therefore equation (\ref{eq-for-U}) has unique solution
\begin{equation*}
U(x) = \frac{1}{2(i+1)}\left<\left(A(x)+A^t(x)\right)\cdot \mathbf x , \mathbf x\right>.
\end{equation*}
Returning to equations (\ref{eq-for-phi-3}) we obtain that $\phi (x)$ satisfies the equation
\begin{equation*}
\begin{split}
E(\phi (x) ) - \phi (x) = f(x)  -
\frac{1}{2(i+1)} \nabla \left<f(x) , \mathbf x \right>
\end{split}
\end{equation*}
where
\begin{equation}
\label{eq-f}
f(x) = \left(A(x) + A^t(x)\right)\cdot \mathbf x.
\end{equation}
\ni Since $f\in \an^{(i+1)}$ we see that this equations have unique solutions
\begin{equation}
\label{eq-final-phi}
\phi (x) = \frac{1}{i}f(x) -
\frac{1}{2\cdot i\cdot (i+1)} \nabla \left<f(x) , \mathbf x \right> .
\end{equation}
Thus equation (\ref{eq-b-coh}) has unique solution $\phi  \in \an^{(i+1)}(\mathcal M_{n\times 1}), \ \ Q\in \an^{(i)}(\mathcal M_{n\times n})$
where $\phi $ is defined by (\ref{eq-final-phi}) and $Q$ is defined by (\ref{eq-for-Q}).
The vector function $f$ is defined by (\ref{eq-f}). We see from (\ref{eq-f}) that
$\vert \vert f\vert \vert < C_1\vert \vert A\vert \vert $. Since  $\left<f(x) , \mathbf x \right>$
is a homogeneous degree $(i+2)$ function, it follows $\vert \vert \nabla \left<f(x) , \mathbf x \right> \vert \vert < C_2(i+2)\vert \vert A\vert \vert $. Here $C_1, C_2$ do not depend on $i$.
 Now we see from (\ref{eq-final-phi}) that
$\vert \vert \phi \vert \vert < \frac{C_3}{i} \vert \vert A\vert \vert $
and using this estimate we see from (\ref{eq-for-Q}) that
$\vert \vert Q\vert \vert < C_4 \vert \vert A\vert \vert $
for some $C_3, C_4$ which do not depend on $i$.

\medskip

{\bf Proof of Proposition \ref{prop-estimates-three-objects} for equation (\ref{eq-c-coh})}
This equation can be expressed in the form
\begin{equation*}
\begin{split}
\phi ^\prime (x) - \left(\phi ^\prime (x)\right)^t - 2Q(x) = A(x) - A^t(x)\\
\left(\phi ^\prime (x) - Q(x)\right)\cdot \mathbf x = A(x)\cdot \mathbf x + h(x)\cdot \mathbf x \\
\text{trace } \phi ^\prime (x) = \text{trace } A(x) + nh(x).
\end{split}
\end{equation*}
We can exclude $Q(x)$ from the first equation
\begin{equation}
\label{eq-for-Q-again}
Q(x) = \frac{1}{2}\left(\phi ^\prime (x) - \left(\phi ^\prime (x)\right)^t - A(x) + A^t(x)\right).
\end{equation}
Substituting to the second and the third equations we obtain the following system for $\phi (x)$ and $h(x)$:
\begin{equation}
\label{eq-CR-1}
\left(\phi ^\prime (x) + \left(\phi ^\prime (x)\right)^t\right)\cdot \mathbf x  =
\left(A(x) + A^t(x)\right)\cdot \mathbf x + 2h(x)\cdot \mathbf x.
\end{equation}
\begin{equation}
\label{eq-CR-2}
\text{trace } \phi ^\prime (x) = \text{trace } A(x) + nh(x).
\end{equation}
Our way of solving this system is as follows. We solve equation (\ref{eq-CR-1}) with respect to $\phi (x)$
in the same way as in the previous subsection. After that (\ref{eq-CR-2}) becomes an equation for
$h(x)$ only.
Thus we work with equation (\ref{eq-CR-1}). Introduce, as in the previous subsection, the function $U(x)$ by (\ref{eq-U}).
In the same way as in the previous section we obtain
\begin{equation}
\label{eq-CR-for-phi-3}
\begin{split}
E(\phi ) - \phi  = - \nabla U(x) + \left(A(x) + A^t(x)\right)\cdot \mathbf x + 2h(x)\cdot \mathbf x.
\end{split}
\end{equation}
Let
\begin{equation}
\label{eq-f-x}
f(x) = \left(A(x)+A^t(x)\right)\cdot \mathbf x.
\end{equation}
We have
\begin{equation}
\label{eq-norm-f}
\vert \vert f\vert \vert < C_1\vert \vert A\vert \vert
\end{equation}
for some $C_1$ which does not depend on $i$.
Taking the inner product of (\ref{eq-CR-for-phi-3}) with the vector $\mathbf x$ we obtain, like in the
previous subsection
\begin{equation*}
\label{eq-CR-for-U}
E(U(x)) - U(x) = \frac{1}{2}\left<f(x)\cdot \mathbf x , \mathbf x\right> + h(x)(x_1^2 + \cdots + x_n^2)
\end{equation*}
and it follows
$$U(x) = \frac{1}{2(i+1)}\left<f(x)\cdot \mathbf x\right> + \frac{1}{i+1}h(x)(x_1^2 + \cdots + x_n^2).$$
Introduce
\begin{equation}
\label{eq-g-x}
g(x) = \nabla \left(\left<f(x)\cdot \mathbf x\right>\right).
\end{equation}
From (\ref{eq-norm-f}) it follows
\begin{equation}
\label{eq-norm-g}
\vert \vert g\vert \vert < C_2\cdot i\cdot \vert \vert A\vert \vert
\end{equation}
for some $C_2$ which does not depend on $i$.
Returning to equation (\ref{eq-CR-for-phi-3}) we obtain
$$E(\phi ) - \phi = f(x) - \frac{1}{2(i+1)}g(x) - \frac{1}{i+1}\nabla \left(h(x)(x_1^2+\cdots + x_n^2)\right) + 2h(x)\cdot \mathbf x.$$
Since
$$\nabla \left(h(x)(x_1^2+\cdots + x_n^2)\right) = (x_1^2 + \cdots + x_n^2)\nabla h(x) + 2h(x)\cdot \mathbf x$$
\ni and $\phi \in \an^{(i+1)}(\mathcal M_{n\times 1})$ we obtain
\begin{equation}
\label{eq-CR-for-phi}
\phi (x) = \frac{1}{i}f(x) - \frac{1}{2i(i+1)}g(x) - \frac{1}{i(i+1)}(x_1^2 + \cdots + x_n^2)\nabla h(x) + \frac{2}{i+1}h\cdot \mathbf x.
\end{equation}
Let $\phi (x) = (\phi _1(x),..., \phi _n(x))^t$ and let
\begin{equation}
\label{eq-CR-s}
\begin{split}
s(x) =  \frac{1}{i}f(x) - \frac{1}{2i(i+1)}g(x) = (s_1(x), ..., s_n(x))^t\\
w(x) = \frac{\partial s_1(x)}{\partial x_1} + \cdots + \frac{\partial s_n(x)}{\partial x_n}.
\end{split}
\end{equation}
From (\ref{eq-norm-f}) and (\ref{eq-norm-g}) it follows
\begin{equation}
\label{eq-norm-w}
\vert \vert w\vert \vert < C_3\vert \vert A\vert \vert
\end{equation}
for some $C_3$ which does not depend on $i$.
We have, for each $k\in \{1,...,n\}$:
$$\frac{\partial \phi _k(x)}{\partial x_k} = \frac{\partial s_k(x)}{\partial x_k}- \frac{2}{i(i+1)}x_k\frac{\partial h}{\partial x_k} -
\frac{1}{i(i+1)}(x_1^2 + \cdots + x_n^2)\frac{\partial ^2h}{\partial x_k^2} + \frac{2}{i+1}x_k\frac{\partial h}{\partial x_k} + \frac{2}{i+1}h $$
and it follows
\begin{equation*}
\begin{split}
\text{trace } \phi ^\prime (x) = + w(x)-\frac{1}{i(i+1)}(x_1^2 + \cdots + x_n^2)\left(\frac{\partial ^2h}{\partial x_1^2} + \cdots + \frac{\partial ^2h}{\partial x_n^2}\right) + \\
+ \frac{2(i-1)}{i(i+1)}E(h(x)) + \frac{2n}{i+1}h(x).
\end{split}
\end{equation*}
Since $h(x)\in \an^{(i)}$ we have $E(h(x)) = ih(x)$ and consequently
$$\text{trace} \phi ^\prime (x) = -\frac{1}{i(i+1)}(x_1^2 + \cdots + x_n^2)\left(\frac{\partial ^2h}{\partial x_1^2} + \cdots + \frac{\partial ^2h}{\partial x_n^2}\right) + \frac{2(i+n-1)}{i+1}h(x) + w(x).$$
Now we return to equation (\ref{eq-CR-2}). Let
$$z :=  w - \text{trace } A \in \an^{(i)}.$$
\ni From (\ref{eq-norm-w}) we obtain
\begin{equation}
\label{eq-norm-z}
\vert \vert z\vert \vert < C_4\vert \vert A\vert \vert
\end{equation}
for some $C_4$ which does not depend on $i$.
Equation (\ref{eq-CR-2}) takes the form
\begin{equation}
\label{eq-CR-for-h}
\begin{split}
\frac{1}{i(i+1)}(x_1^2 + \cdots + x_n^2)\left(\frac{\partial ^2h(x)}{\partial x_1^2} + \cdots + \frac{\partial ^2h(x)}{\partial x_n^2}\right) + \frac{(n-2)(i-1)}{i+1}h(x) = z(x)
\end{split}
\end{equation}
Now we need the following statement.

\begin{lem}
\label{prop-CR-eq-h}
Let $n\ge 3$ and $i\ge 2$. For any given $z\in \an^{(i)}$, equation (\ref{eq-CR-for-h}) has a unique solution $h\in \an^{(i)}$ and one has $\vert \vert h\vert \vert < C\vert \vert z\vert \vert $ where the constant $C$ does not depend on $i$.
\end{lem}

\medskip

\ni This proposition is proved below. Proposition \ref{prop-estimates-three-objects} for equation (\ref{eq-c-coh}) and  $i\ge 2$ is a direct corollary of Proposition \ref{prop-CR-eq-h}, the formula (\ref{eq-CR-for-phi}) expressing $\phi $ in terns of $h$, the formula (\ref{eq-for-Q-again}) expressing $Q$ in terms of $\phi $, and the
estimates (\ref{eq-norm-z}), (\ref{eq-norm-f}), (\ref{eq-norm-g}). Indeed, $w$ is known from $A$, so is $z$ with estimate $(\ref{eq-norm-z})$. According to the previous proposition, $h$ solves $(\ref{eq-CR-for-h})$ with estimate $\|h\|\leq \tilde C \|A\|$. Since $s$ is known from $A$, so is $\phi$ ($\ref{eq-CR-for-phi}$) and it satisfies $\|\phi\|\leq \frac{C}{i}\|A\|$. Finally, $Q$ is known from $A$ and satisfies $\|Q\|\leq  \|A\|$ and so does $R$.
To complete the proof of Proposition \ref{prop-estimates-three-objects} it remains to prove
Lemma \ref{prop-CR-eq-h}.

\medskip

{\bf Proof of Lemma \ref{prop-CR-eq-h}}
Consider the linear operator
$L_i: \ \an^{(i)}\to \an^{(i)}$,
$$L_i(h) = \frac{1}{i(i+1)}(x_1^2 + \cdots + x_n^2)\left(\frac{\partial ^2h(x)}{\partial x_1^2} + \cdots + \frac{\partial ^2h(x)}{\partial x_n^2}\right) + \frac{(n-2)(i-1)}{i+1}h(x).$$
The key point is that the operator $L_i$ is selfadjoint with respect to Belitskii scalar product.
It is easy to see that this property of $L_i$ reduces the lemma to the following statement:
the eigenvalues of $L_i$
are bigger than a positive constant which does not depend on $i$.
We will prove that all eigenvalues of $L_i$ are bigger than $1/2$.

\medskip

Consider any eigenvector  $h(x)\in \mathcal M_{1\times 1}^{(i)}$ of $L_i$ corresponding to
eigenvalue $\lambda $:
$$\frac{1}{i(i+1)}(x_1^2 + \cdots + x_n^2)\left(\frac{\partial ^2h(x)}{\partial x_1^2} + \cdots + \frac{\partial ^2h(x)}{\partial x_n^2}\right) = \left(\lambda - \lambda _1\right)h(x).$$
$$\lambda _1 = \frac{(n-2)(i-1)}{i+1}.$$
If
$\frac{\partial ^2h(x)}{\partial x_1^2} + \cdots + \frac{\partial ^2h(x)}{\partial x_n^2} =0$
then
$\lambda = \lambda _1.$
If $\frac{\partial ^2h(x)}{\partial x_1^2} + \cdots + \frac{\partial ^2h(x)}{\partial x_n^2}\ne 0$
then $h(x)$ must have the form
$h(x) = (x_1^2 + \cdots + x_n^2)\widetilde h(x), \ \ \widetilde h(x)\in \mathcal M_{1\times 1}^{(i-2)}$ and it is easy
to compute that $\widetilde h(x)$ satisfies the equation
$$ \frac{1}{i(i+1)}(x_1^2 + \cdots + x_n^2)\left(\frac{\partial ^2\widetilde h(x)}{\partial x_1^2} + \cdots + \frac{\partial ^2\widetilde h(x)}{\partial x_n^2}\right) = \left(\lambda - \lambda _1 - \lambda _2\right)\widetilde h(x),$$
$$\lambda _2 = \frac{4(i-2)+2n}{i(i+1)}. $$
 If $\frac{\partial ^2\widetilde h(x)}{\partial x_1^2} + \cdots + \frac{\partial ^2\widetilde h(x)}{\partial x_n^2}= 0$ then
$\lambda = \lambda _1 + \lambda _2 > \lambda _1.$
 If $\frac{\partial ^2\widetilde h(x)}{\partial x_1^2} + \cdots + \frac{\partial ^2\widetilde h(x)}{\partial x_n^2}\ne 0$ then $\widetilde h(x)$ must have the form
$$\widetilde h(x) = (x_1^2 + \cdots + x_n^2)\widehat h(x), \ \ \widehat h(x)\in \mathcal M_{1\times 1}^{(i-4)}.$$
In this case $\widehat h(x)$ satisfy the equation
$$ \frac{1}{i(i+1)}(x_1^2 + \cdots + x_n^2)\left(\frac{\partial ^2\widehat h(x)}{\partial x_1^2} + \cdots + \frac{\partial ^2\widehat h(x)}{\partial x_n^2}\right) = \left(\lambda - \lambda _1 - \lambda _2-\lambda _3\right)\widehat h(x),$$
$$\lambda _3 = \frac{4(i-4)+2n}{i(i+1)}. $$
Continuing in the same way we come to conclusion
$\lambda \ge \lambda _1$
for any eigenvalue $\lambda $ of $L_i$. Since $n\ge 3$ and $i\ge 2$ we have
$\lambda \ge \frac{1}{2}$
for any eigenvalue $\lambda $ of $L_i$.

\end{appendices}

\bibliography{newsumbib}

\def\cprime{$'$} \def\cprime{$'$}

\end{document}